\RequirePackage[l2tabu,orthodox]{nag} 
\documentclass[11pt,headings=standardclasses,abstract=true]{scrartcl}
\usepackage[utf8]{inputenc}
\usepackage{ulem}
\usepackage[margin=30mm,right=37mm]{geometry}
\usepackage{authblk}
\usepackage[T1]{fontenc}        
\usepackage{lmodern}            
\usepackage[intlimits]{amsmath} 

\usepackage{hyperref}           
\usepackage{grffile}            

\usepackage[ugly]{units}        

\usepackage{url}                
\usepackage{xspace}             
\usepackage{xcolor}             
\usepackage{booktabs}           

\usepackage{listings}           

\usepackage[textwidth=30mm]{todonotes}
\usepackage{enumitem}
\usepackage{mathtools}
\usepackage{verbatim}
\usepackage{accents}

\usepackage{kbordermatrix}

\usepackage{amssymb}

\usepackage{amsthm}

\newtheorem{theorem}{Theorem}
\newtheorem{definition}[theorem]{Definition}
\newtheorem{ass}[theorem]{Assumption}
\newtheorem{corollary}[theorem]{Corollary}
\newtheorem{mylemma}[theorem]{Lemma}
\newtheorem{proposition}[theorem]{Proposition}
\theoremstyle{remark}
\newtheorem{remark}[theorem]{Remark}

\if 11
\newcommand\myopbf\bf                     
\newcommand\myboldmath\boldsymbol
\newcommand{\vect}[1]{{\mathbf {#1}}} 		
\else
\newcommand\myopbf\rm                           
\newcommand{\vect}[1]{#1}		
\newcommand\myboldmath\unboldmath
\fi

\newcommand*{\mydot}[1]{%
  \accentset{\mbox{\large\bfseries .}}{#1}}

\if 0 
\newcommand{\Atop}[1][{}]{\vect{A}_{\!#1}}
\newcommand{\jcurrv}[1][{}]{\myboldmath\jmath_{#1}}  %
\newcommand{\npot}[1][{}]{{\vect{e}}_{{#1}}}
\newcommand{\charge}[1][{}]{{\vect{q}}_{{#1}}}

\newcommand{\flux}[1][{}]{\myboldmath\phi_{{#1}}}
\newcommand{\volt}[1][{}]{{\vect{v}}_{{#1}}}
\newcommand{\vsrc}[1][{}]{{\vect{v}}_{{#1}}}
\newcommand{\csrc}[1][{}]{{\myboldmath\imath}_{{#1}}}
\newcommand{\resi}[1][{}]{{\vect{g}}}

\newcommand{\zeros}{\vect{0}}

\else
\newcommand{\Atop}[1][{}]{{A}_{#1}}
\newcommand{\jcurrv}[1][{}]{\jmath_{#1}}
\newcommand{\djcurrv}[1][{}]{\mydot{\jmath}_{#1}}
\newcommand{\tjcurrv}[1][{}]{\tilde{\jmath}_{#1}}
\newcommand{\npot}[1][{}]{{{e}}_{{#1}}}
\newcommand{\charge}[1][{}]{{q}_{{#1}}}

\newcommand{\flux}[1][{}]{\phi_{{#1}}}
\newcommand{\volt}[1][{}]{{{v}}_{{#1}}}
\newcommand{\vsrc}[1][{}]{{{v}}_{{#1}}}
\newcommand{\csrc}[1][{}]{{\imath}_{{#1}}}
\newcommand{\resi}[1][{}]{{{g}}}

\newcommand{\zeros}{{0}}
\renewcommand{\vect}[1]{#1}		
\renewcommand\myboldmath\unboldmath
\fi

\newcommand{\interval}{\mathcal{I}}

\newcommand{\tim}{t}

\newcommand{\Cnet}{\vect{C}}
\newcommand{\Lnet}{\vect{L}}

\newcommand{\tr}{\top}

\usepackage{bbm}
\newcommand{\real}{\mathbbm{R}}
\newcommand{\diag}{\text{diag}}
\newcommand{\blkdiag}{\text{blkdiag}}

\newcommand{\diff}{\,\textrm{\normalfont d}}

\newcommand{\ddt}{\frac{\diff}{\diff t}}

\DeclareMathOperator{\init}{init}
\DeclareMathOperator{\ter}{ter}
\DeclareMathOperator{\im}{im}

\newcommand{\lambdav}{\myboldmath\lambda}
\newcommand{\tot}{\scriptstyle\text{tot}}

\newcommand{\numsys}{k}
\definecolor{darkgreen}{rgb}{0.15,0.65,0.15}
\newcommand{\myendofenv}{\qed}

\begin{document}

\title{Dynamic iteration schemes and port-Hamiltonian formulation in coupled DAE circuit simulation}

\author[1]{Michael~G\"unther\thanks{guenther@uni-wuppertal.de}}
\author[1]{Andreas~Bartel\thanks{bartel@uni-wuppertal.de}}
\author[1]{Birgit~Jacob\thanks{bjacob@uni-wuppert
al.de}}
\author[2]{Timo~Reis\thanks{timo.reis@uni-hamburg.de}}
\affil[1]{Bergische Universität Wuppertal\\ 
Fakult\"at für Mathematik und Naturwissenschaften\\
Gau\ss straße 20\\
42119 Wuppertal, Germany}
\affil[2]{Universit\"at Hamburg\\ Fachbereich Mathematik\\ Bundesstra\ss e~55\\ 20146 Hamburg, Germany}

\renewcommand\Affilfont{\itshape\small}

\date{}
\maketitle

\begin{abstract}
Electric circuits are usually described by charge- and flux-oriented modified nodal analysis.
In this paper, we derive models as port-Hamiltonian systems on several levels: overall systems, multiply coupled systems and systems within dynamic iteration procedures. To this end,  we introduce 
 new classes 
of port-Hamiltonian differential-algebraic equations.
 Thereby, we additionally \color{black} 
allow for nonlinear dissipation on a subspace of the state space. Both, each subsystem and the overall system, possess a port-Hamiltonian structure.  A structural analysis is performed for the new setups. \color{black}
Dynamic iteration schemes are investigated and we show that the Jacobi approach as well as an adapted Gauss-Seidel approach lead to port-Hamiltonian differential-algebraic equations.
\color{black}
\end{abstract}

\hspace*{0.5cm}


\hspace*{0.5cm}


\section{Introduction}
Models for electric circuits are based on a collection of basic electric components. These form 
edges of a directed graph. 
The directed graph represents the interconnection structure, which is represented by the incidence matrix $\Atop[]$  that enables to formulate Kirchhoff's voltage law (KVL) and Kirchhoff's current law (KCL).  
Electric components describe a certain electric effect. 
In our case, these are 
resistances, capacitances, inductances, independent current and independent voltage sources. \color{black}
%

An oftentimes used modeling approach to electric circuits is the 
modified nodal analysis (MNA), see \textsc{McCalla} \cite{McCalla_1988}. 
For charge and flux conservation, this is extended to the charge/flux-oriented form, see \textsc{G\"unther \& Feldmann} \cite{Guenther_1999}.
%
Now, the KVL allows the assignment of vertex potentials (often referred tp as node potentials) $\npot{}$ to each vertex except for the grounded one which has a given value. 
Apart for the vertex potentials, one has as unknowns the currents through inductances $\jcurrv[L]$ and through voltages sources $\jcurrv[V]$, 
%
the charges $\charge[C]$ at the capacitances and the magnetic fluxes $\flux[L]$ at the inductances.
\color{black}
%
Thus,  
the vector of unknowns reads
\[
	x^\tr(t)=\bigl(\npot{}^\tr (t),\, \jcurrv[L]^\tr (t),\, \jcurrv[V]^\tr (t),\, \charge[C]^\tr (t),\, \flux[L]^\tr (t)\bigr) \in \real^{d}, 
\]
\color{black}
where the time $t$ evolves in a~specified operation interval $\interval:=[0, t_e]\subseteq \real$.
The circuit can now be described by the equations of {\em charge/flux-oriented modified nodal analysis (MNA)}, which reads
\begin{subequations}\label{eq:mna}
	\begin{align}
	\label{eq:kcl}
	\Atop[C]\frac{\mathrm{d}}{\mathrm{d} t} \charge[C]
	+\Atop[R] \resi[R](\Atop[R]^{\tr} \npot{})
	+\Atop[L] \jcurrv[L]
	+\Atop[V] \jcurrv[V]
	+\Atop[I] \csrc(\tim)
	& = \zeros,\\
    \frac{\mathrm{d}}{\mathrm{d} t} \flux[L] - \Atop[L]^{\tr} \npot{}
	& =  \zeros, \label{eq:flux} \\
	\Atop[V]^{\tr} \npot{} - \vsrc(\tim)
	& =  \zeros , 	\label{eq:voltsrc}\\
\charge[C]-\charge(\Atop[C]^{\tr}\npot)&=\zeros,\\
\flux[L]-\flux(\jcurrv[L])&=\zeros,
	\end{align}
\end{subequations}
where we have component-specific incidences matrices $\Atop[\star]$. 
Moreover, we use for the component relations:
$\charge(\volt)$ for capacitances, $\resi(\volt)$ for resistances,  $\flux(\jcurrv[L])$ for inductances,
$\vsrc(t)$ for independent voltage sources and $\csrc(t)$ for independent current sources, where the latter two variables are given beforehand. The involved matrices and functions are further specified in the forthcoming Section~\ref{sec:strucana}.

One aim of this paper is to model the MNA as port-Hamiltonian DAE. Port-Hamiltonian systems  form a joint structure of systems in various physical domains.
This approach has its roots in analytical mechanics and starts from the principle of least action, and proceeds towards the Hamiltonian equations of motion.
Dynamic systems, which result from variational principles, can usually be modeled by a port-Hamiltonian system. A system theoretical and geometric treatment of port-Hamiltonian ordinary differential systems goes back to \textsc{van der Schaft}  and  there is by now a well-established theory
(see \textsc{van der Schaft}~\cite{vdS04} and \textsc{Jeltsema \& van der Schaft}~\cite{JvdS14} for an overview), 
which has been applied to electrical circuits, \textsc{Gernandt} et al.~\cite{GeHaReVdS20}.  
Only recently the concept has been generalized to port-Hamiltonian differential-algebraic systems, that is, ordinary differential equations with algebraic constrains, 
(see \textsc{van der Schaft}~\cite{vdS13}, \textsc{Maschke \& van der Schaft}~\cite{MvdS18,MaVdS19}). 
In \textsc{Beattie} et al.~\cite{Beattie_2018},
linear time-varying port-Hamiltonian differential-algebraic systems have been studied and the notion has been generalized to quasilinear systems in \textsc{Mehrmann \& Morandin}~\cite{MehM19ppt}. 

Now, we extend the class even further in order to allow for nonlinear dissipation on a subspace of the state space. We introduce two circuit models throughout this article, which are slightly different from the charge/flux oriented MNA~\eqref{eq:mna}. Both models are formulated as port-Hamiltonian DAE. Furthermore, we investigate 
 multiply coupled circuits and extend our definitions in this respect to multiply coupled 
port-Hamiltonian DAEs. In fact,  we show that port-Hamiltonian DAEs can be coupled in such a way that the overall system is a port-Hamiltonian DAE as well. This is applied 
to our circuits models.

 A further 
novelty of this paper is the study of dynamic iteration schemes
in the context of port-Hamiltonian systems. For an overview on dynamic iteration schemes for ODEs, see \textsc{Burrage}~\cite{Burrage_1995}. These schemes have also been studied for DAEs, where convergence cannot be generally guaranteed, see e.g. \textsc{Lelarsemee} et al.~\cite{LRS82}, \textsc{Jackiewicz \& Kwapisz}~\cite{ZK96}
and \textsc{Arnold \& G\"unther}~\cite{AG01}.
Here, we investigate dynamic iteration schemes for coupled systems composed by $\numsys$ subsystems with dedicated coupling equation. For these type systems, we show that both,
Jacobi- and Gauss-Seidel type schemes
can be interpreted as port-Hamiltonian systems. In order to achieve this goal we have modify slightly the interconnections. Again as an example we study electric circuits.

The outline of the paper is as follows: Section~\ref{sec:strucana} addresses the mathematical modeling background for the charge/flux oriented circuit equations. In the following Section~\ref{sec:phs.for.networks} the various port-Hamiltonian formulations are introduced. Then, a DAE index analysis is performed for our models (Section~\ref{sec:IndAna}). Section~\ref{sec:network-network} introduces structural properties for coupled circuits and Section~\ref{sec:pHsys} merges the port-Hamiltonian formulation with the dynamic iteration schemes. Finally, there are conclusions. 

\section{Circuit equations - a structural analysis}\label{sec:strucana}

We will consider special variants of the charge/flux-oriented MNA equations \eqref{eq:mna},  suitable for the port-Hamiltonian setting. To this end, we first present some fundamentals on circuit equations. An electrical circuit is described by the properties of its components together with the interconnection structure. The latter is modelled by a~(loop-free, directed and finite) graph. Moreover, many properties of the circuit equations such as soundness, passivity and DAE-index, depend both on topological conditions of the underlying graph, as for instance about the absence of certain component-specific cycles and and cuts (see e.g.~\textsc{Bartel et al.}~\cite{BBS11} and \textsc{Bartel \& G\"unther}~\cite{BaGu2018}). To this end, we need some preliminaries from graph theory, see e.g.~\cite{Dies17}.

\begin{definition}[Graphs and subgraphs]
A \emph{directed graph} is a tuple $\mathcal{G} = (V,E,\init,\ter)$ consisting of a \emph{vertex set} $V$, a \emph{edge set} $E$ and two maps $\init, \ter: E \rightarrow V$ as\-signing to each edge $e$ an {\em initial vertex} $\init(e)$ and a {\em terminal vertex} $\ter(e)$. The edge $e$ is said to be {\em directed from $\init(e)$ to $\ter(e)$}. $\mathcal{G}$ is said to be {\em loop-free}, if $\init(e)\neq\ter(e)$  for all $e\in E$. Let $V' \subset V$ and $E' \subset E$ with
\[ E' \subset \left. E\right|_{V'} := \{ e \in E : \init(e)\in V'\,\wedge\, \ter(e) \in V' \}. \]
Then the triple $(V',E',\left. \init\right|_{E'},\left. \ter\right|_{E'})$ is called a \emph{subgraph of $\mathcal{G}$}. If $E' = \left. E\right|_V'$, then the subgraph is called the \emph{induced subgraph} on $V'$.
If $V' = V$, then the subgraph is called \emph{spanning}.
Additionally a \emph{proper subgraph} is one where $E' \neq E$.
$\mathcal{G}$ is called {\em finite}, if $V$ and $E$ are finite.
\end{definition}
The notion of a~{\em path} in a~directed graph $\mathcal{G} = (V,E,\init,\ter)$ is quite descriptive. However, since a~path may also go through an edge in reverse direction, we define for each $e \in E$ an additional edge $-e \not\in E$ with $\init(-e)=\ter(e)$ and $\ter(-e)=\init(e)$.
\begin{definition}[Paths, connected, cycles, cuts]
Let $\mathcal{G} = (V,E,\init,\ter)$ be a~finite directed graph and let $\mathcal{K} = (V,E', \left. \init\right|_{E'}, \left. \ter\right|_{E'})$ be a~spanning subgraph.\\
A~$r$-tuple $e = (e_1,\ldots,e_r) \in ({E}\cup -E)^r$ is called a~\emph{path from $v$ to $w$}, if the initial vertices $\init(e_1),\ldots,\init(e_r)$ are distinct, $\ter(e_i) = \init(e_{i+1})$ for all $i \in \{ 1,\ldots, r-1 \}$, as well as $\init(e_1)=v$ and $\ter(e_r)=w$.\\
A \emph{cycle} is a~path from $v$ to $v$. Two vertices $v,w$ are \emph{connected}, if there exists a path from $v$ to $w$. This gives an equivalence relation on the vertex set. The induced subgraph on an equivalence class of connected vertices gives a \emph{component} of the graph. A graph is called {\em connected}, if there is only one component.\\
$\mathcal{K}$ is called a \emph{cut} of $\mathcal{G}$, if $\mathcal{G} - \mathcal{K} := (V, E \backslash E', \left.\init\right|_{E \backslash E'}, \left.\ter\right|_{E \backslash E'})$ has two connected components.
\end{definition}
In the context of electrical circuits, finite and loop-free directed graphs are of major importance. These allow to associate a~special matrix, see \textsc{Andr\'asfai}~\cite[Sec.~3.2]{And91}.
\begin{definition}[Incidence matrix]
Let $\mathcal{G}=(V,E,\init,\ter)$ be a finite and loop-free directed graph. Let $E=\{e_1,\ldots,e_m\}$ and $V=\{v_1,\ldots,v_n\}$. Then the \emph{all-vertex incidence matrix} of $\mathcal{G}$ is $A_0\in\real^{n\times m}$ with
\[a_{jk}=\begin{cases}1&\init(e_k)=v_j, \\-1&\ter(e_k)=v_j,\\0&\text{otherwise.}\end{cases}\]
\end{definition}
If $\mathcal{G}$ is connected, then the co-rank of $A_0$ equals one, whence the deletion of an arbitrary row leads to a~matrix with full row rank \cite[p.\ 140]{And91}. In the context of electrical circuits, this corresponds to the grounding of this vertex. 

Starting with an incidence matrix $A$ of a~finite and loop-free directed graph $\mathcal{G}$, along with a spanning subgraph $\mathcal{K}$ of $\mathcal{G}$, it is possible to obtain an incidence matrix of $\mathcal{K}$ by deleting all columns corresponding to edges of $\mathcal{G} - \mathcal{K}$. By rearranging the columns, it follows that the matrix $A$ is of the form
\begin{equation}\label{sorted}
A = [A_{\mathcal{G}-\mathcal{K}} \: A_{\mathcal{K}} ].
\end{equation}
Next we collect some auxiliary results on incidence matrices corresponding to subgraphs from \textsc{Est\'evez Schwarz \& Tischendorf}~\cite{Estevez-Schwarz_2000aa}. Note that this reference has wording which slightly differs from ours, as, for instance, cycles are called {\em loops} therein. Our notation is oriented by the standard reference \textsc{Diestel}~\cite{Dies17} for graph theory. The first statement of the following proposition can be inferred from the fact that incidence matrices of connected (sub-)graphs have full row rank. The further assertions are shown in \cite{Estevez-Schwarz_2000aa}.
\begin{proposition}\cite[Thm.\ 2.2]{Estevez-Schwarz_2000aa}\label{graphalg} 
Let $\mathcal{G}$ be a finite and loop-free connected graph with incidence matrix $A$ and let $\mathcal{K}$ be a spanning subgraph.
Assume that the incidence matrix is partitioned as in \eqref{sorted}.
Moreover, let $\mathcal{L}$ be a spanning subgraph of $\mathcal{K}$, and, likewise, 
that $A_{\mathcal{K}}$ is partitioned as 
\begin{equation}\label{sorted2}
A_{\mathcal{K}} = [A_{\mathcal{K} - \mathcal{L}}\, A_{\mathcal{L}} ].
\end{equation}
Then the following holds:
\begin{enumerate}[label = (\roman*)]
 \item $\mathcal{G}$ does not contain any cuts only consisting of edges in $\mathcal{K}$ if, and only if, $\ker A^\top_{\mathcal{G} - \mathcal{K}} = \{0\}$. 
 \item $\mathcal{G}$ does not contain any cycles   only consisting of edges in $\mathcal{K}$ if, and only if, $\ker A_{\mathcal{K}} = \{0\}$. 
 \item $\mathcal{G}$ does not contain any cycles  only consisting of edges in $\mathcal{K}$ except for cycles only consisting of edges in $\mathcal{L}$ if, and only if, 
 \[\{x\in\real^{n_{\mathcal{K} - \mathcal{L}}}\,|\, A_{\mathcal{K} - \mathcal{L}}x\in\im A_{\mathcal{L}}\}=\{0\}.\] 
\end{enumerate}
\end{proposition}

When considering an electrical circuit as a~graph, we can split the incidence matrix into submatrices respectively representing the columns to capacitances, resistances, inductances, voltage sources and current sources, i.e.,
\[(\Atop[C]\, \Atop[R]\, \Atop[L]\, \Atop[I]\, \Atop[V] ).\]
In other words, we consider the incidence matrices of the spanning subgraphs formed by specific electrical components.
Now we are able to formulate our assumptions on the circuit.

%
%

%

\begin{ass}
\label{ass:soundness-passivity}
%
\begin{enumerate}[leftmargin=.25in]\
 \item[a)] 
 \textbf{Soundness.} The circuit graph has at least one edge and is connected.
The circuit graph further neither contains cycles consisting only of edges of voltage sources nor cuts consisting only of edges of current sources. Equivalently, by Proposition~\ref{graphalg},  
 $\Atop[V]$ and $\displaystyle(\Atop[C]\, \Atop[R]\, \Atop[L]\, \Atop[V] )^{\tr}$ have full column rank.
\item[b)] \textbf{Passivity.} The functions $\charge$, $\flux$ and $\resi$ fulfill
\begin{itemize}
\item[(i)] $\charge:\real^{n_C}\to\real^{n_C}$ and $\flux:\mathbbm{R}^{n_L}\to\mathbbm{R}^{n_L}$ are bijective, continuously differentiable, and their Jacobians 
\[
    C(u_C):=\,\frac{\partial \charge}{\partial u_C}(u_C),
    \qquad
    L(\jcurrv[L]):=\,\frac{\partial \flux}{\partial          
                \jcurrv[L]}(\jcurrv[L])
\]
are symmetric and positive definite for all $u_C\in \real^{n_C}$, $\jcurrv[L]\in \real^{n_L}$.
\item[(ii)] $\resi:\mathbbm{R}^{n_R}\to\mathbbm{R}^{n_R}$ is continuously differentiable, and its Jacobian has the property that $\frac{\partial g}{\partial u_R}(u_R)+ \frac{\partial g}{\partial u_R}(u_R)^{\tr}$ is positive definite for all $u_R\in\real^{n_R}$.
\end{itemize}
\end{enumerate}
\end{ass}
The condition on the charge and flux functions imply that there exist certain scalar-valued functions which will later on be shown to be expressing the energy of an electrical circuit.
\begin{proposition}\label{prop:energy}
If $\charge:\real^{n_C}\to\real^{n_C}$ and $\flux:\mathbbm{R}^{n_L}\to\mathbbm{R}^{n_L}$ fulfill Assumption~\ref{ass:soundness-passivity}b)(i), then there exist twice continuously differentiable functions 
$V_C:\real^{n_C}\to\real$, $V_L:\real^{n_L}\to\real$ with the following properties:
\begin{itemize}
\item[(a)] $V_C:\real^{n_C}\to\real$, $V_L:\real^{n_L}\to\real$ are strictly convex, that is,
\begin{align*}
&\forall\, \lambda\in[0,1]:\\
& 
    \forall\, \charge[C,1],\charge[C,2]\in\real^{n_C}: \; V_C(\lambda\charge[C,1]+(1-\lambda)\charge[C,2])<
\lambda V_C(\charge[C,1])+(1-\lambda)V_C(\charge[C,2]),
\\
& 
    \forall\, \flux[L,1],\flux[L,2]\in\real^{n_L}: \; 
        V_L(\lambda\flux[L,1]+(1-\lambda)\flux[L,2])<
        \lambda V_L(\flux[L,1])+(1-\lambda)V_L(\flux[C,2]),
\end{align*}
\item[(b)] The gradients of $V_C$ and $V_L$ are, respectively, the inverse functions of $\charge$ and $\flux$. That is, 
\[\begin{aligned}
&\forall\, \charge[C]\in\real^{n_C}:\; &\nabla V_C(\charge[C])=&\charge^{-1}(\charge[C]),\\
&\forall\, \flux[L]\in\real^{n_L}:\; &\nabla V_L(\charge[L])=&\flux^{-1}(\flux[L]).
\end{aligned}\]
\item[(c)] $V_C$ and $V_L$ take, except for one $\charge[C]^*\in\real^{n_C}$ (resp.~$\flux[L]^*\in\real^{n_L}$), positive values. That is, there exist $\charge[C]^*\in\real^{n_C}$, $\flux[L]^*\in\real^{n_L}$ such that $V_C(\charge[C])>0$ and $V_L(\flux[L])>0$ for all $\charge[C]\in\real^{n_C}\setminus\{\charge[C]^*\}$ and
$\flux[L]\in\real^{n_L}\setminus\{\flux[L]^*\}$.
\end{itemize}
\end{proposition}
\begin{proof}
 By changing the roles of fluxes and charges, it suffices to prove the statement only for the charge function.\\
 Since $\charge$ is bijective and its derivative is, by positive definiteness of $C(u_C)$, invertible, the inverse function $\charge$ is continuously differentiable as well, and the Jacobian reads
\[
\frac{d\charge^{-1}}{\charge[C]}(\charge[C])=C(\charge^{-1}(\charge[C]))^{-1}.
\]
In particular, the Jacobian of $\charge^{-1}$ is pointwise symmetric and positive definite as well. This together with the trivial fact that $\real^{n_C}$ is simply connected implies that there exists some twice differentiable function $V_C:\real^{n_C}\to\real$ with $\nabla V_C(\charge[C])=\charge^{-1}(\charge[C])$ for all $\charge[C]\in\real^{n_C}$. The pointwise positive definiteness of 
$\frac{d\charge^{-1}}{\charge[C]}(\charge[C])$ implies that $V_C$ is strictly convex. Hence, $V_C$ has a~unique minimum
$\charge[C]^*\in\real^{n_C}$. Now replacing $V_C$ with the difference of $V_C$ and $V_C(\charge[C]^*)$, this function has the desired properties, and the proof is complete.
\end{proof}

\begin{remark}\label{rem:chargefluxbijective}\
\begin{itemize}
\item[(a)]
If $n_C=n_L=n_R=1$, then the conditions on $\charge$, $\flux$ and $\resi$ imply that these functions are strictly monotonically increasing with
    \[\lim_{u_C\to\pm\infty}\charge(u_C)=\pm\infty,\qquad
    \lim_{\jcurrv[L]\to\pm\infty}\flux(\jcurrv[L])=\pm\infty,
    \qquad \lim_{u_R\to\pm\infty} g(u_R)=\pm\infty.
    \]
\item[(b)]
 Bijectivity of $\charge$, $\flux$ might by difficult to check. A~sufficient condition can be inferred from the Hadamard-Levy Theorem \cite{Pla74}, which gives bijectivity of $\charge$ and $\flux$, if the conditions
\[\int_0^\infty\min_{\|u_C\|=r}\|C(u_C)^{-1}\|^{-1}=\infty,\quad
\int_0^\infty\min_{\|\jcurrv[L]\|=r}\|L(\jcurrv[L])^{-1}\|^{-1}=\infty.
\]
are fulfilled. By using the positive definiteness of $C(u_C)$ and $L(\jcurrv[L])$, the latter is equivalent to
\[\int_0^\infty\min_{\|u_C\|=r}\lambda_{\min}(C(u_C))=\infty,\qquad
\int_0^\infty\min_{\|\jcurrv[L]\|=r}\lambda_{\min}(L(\jcurrv[L]))=\infty,
\]
where $\lambda_{\min}$ denotes the smallest eigenvalue of a~matrix.
\myendofenv
\end{itemize}
\end{remark}
We will discuss two circuit model throughout this article, which are slightly different from the charge/flux oriented MNA \eqref{eq:mna}. Both models are formulated such that they fit into the PH-DAE framework introduced in Section~\ref{sec:phs.for.networks}.

 The first model is based on using both component equations for charges and fluxes: for the fluxes, we apply $\flux^{-1}$ to the equation $\flux[L]-\flux(\jcurrv[L])=0$ to obtain $\jcurrv[L]=\flux^{-1}(\flux[L])$ which is further eliminated. Likewise, $\charge^{-1}$ is applied to the equation $\charge[C]- \charge(\Atop[C]^{\tr} \npot)$ for the charges, which results into $\Atop[C]^{\tr}\npot{}-\charge^{-1}(\charge[C])$. Summing up, we get 
\color{black}
\begin{subequations}
\label{model1}
\begin{align}
    \nonumber
    \ddt\begin{pmatrix}
     \Atop[C] & 0 & 0 & 0 \\
        0 & I & 0 & 0 \\
        0 & 0 & 0 & 0 \\
        0 & 0 & 0 & 0 \\
    \end{pmatrix}
    \!\!
    \begin{pmatrix}
    \charge[C] \\ \flux[L] \\ \npot{} \\ \jcurrv[V]     \end{pmatrix}
    & =
    \begin{pmatrix}
         0 & -\Atop[L] & 0 & -\Atop[V] \\
         \Atop[L]^{\tr} & 0 & 0 & 0 \\
         0 & 0 & 0 & 0 \\
        \Atop[V]^{\tr} & 0 & 0 & 0 \\
    \end{pmatrix}
    \begin{pmatrix}
    \npot{} \\  \flux^{-1}(\flux[L]) \\ \charge^{-1}(\charge[C]) \\  \jcurrv[V]
    \end{pmatrix} 
    \\
    & \label{eq:mna.cond_a}
    \qquad -
    \begin{pmatrix}
    \Atop[R]g(\Atop[R]^{\tr}\npot{})\\0\\ 
        \Atop[C]^\tr \npot - \charge^{-1}(\charge[C]) \\0
    \end{pmatrix}+ \begin{pmatrix}
    -\Atop[I] & 0  \\ 0 & 0 \\ 0 & 0 \\ 0 & - I
    \end{pmatrix}
    \!\!\begin{pmatrix}
    \csrc({\tim}) \\ \vsrc({\tim})
    \end{pmatrix}\! ,
\end{align}
and output equation
\begin{align}
& y= 
\begin{pmatrix}
    -\Atop[I] & 0  \\ 0 & 0 \\ 0 & 0 \\ 0 & - I
    \end{pmatrix}^{\tr}
    \begin{pmatrix}
    \npot{} \\  \flux^{-1}(\flux[L]) \\ \charge^{-1}(\charge[C]) \\  \jcurrv[V]
    \end{pmatrix} 
=
\begin{pmatrix} -\Atop[I]^{\tr} e  \\ -\jcurrv[V]\end{pmatrix}.
 \end{align}
 \end{subequations}

In the second model, we further add the variable $\jcurrv[C]$ and the equation $\ddt \charge[C]=\jcurrv[C]$ to the model \eqref{model1}. Moreover, the expression $\ddt \charge[C]$ in the first equation of \eqref{model1} is replaced by $\jcurrv[C]$, which results into 
\begin{subequations}
\label{model2}\begin{align}
    \nonumber
    \ddt\begin{pmatrix}
     0 & 0 & 0 & 0 & 0 \\
     0 & 0 & 0 & 0 & 0 \\
     0 & 0 & I & 0 & 0 \\
     0 & 0 & 0 & I & 0 \\
     0 & 0 & 0 & 0 & 0 \\
    \end{pmatrix}
    \!\!
    \begin{pmatrix}
    \npot{} \\ {\jcurrv[C]} \\ {\charge[C]} \\ {\flux[L]} \\ {\jcurrv[V]}
    \end{pmatrix}
    & =
    \begin{pmatrix}
         0 & -\Atop[C] & 0 & -\Atop[L] & -\Atop[V] \\
        \Atop[C]^{\tr} & 0 & -I & 0 & 0 \\
        0 & I & 0 & 0 & 0 \\
        \Atop[L]^{\tr} & 0 & 0 & 0 & 0 \\
        \Atop[V]^{\tr} & 0 & 0 & 0 & 0 \\
    \end{pmatrix}
    \begin{pmatrix}
    \npot{} \\  \jcurrv[C] \\ \charge^{-1}(\charge[C]) \\  \flux^{-1}(\flux)\\  \jcurrv[V]
    \end{pmatrix} 
    \\
    & \label{eq:mna.cond_b}
    \qquad -
    \begin{pmatrix}
    \Atop[R]g(\Atop[R]^{\tr}\npot{})\\0\\0\\0\\0
    \end{pmatrix}+ \begin{pmatrix}
    -\Atop[I] & 0 \\ 0 & 0 \\ 0 & 0 \\ 0 & 0 \\ 0 & - I
    \end{pmatrix}
    \!\!\begin{pmatrix}
    \csrc({\tim}) \\ \vsrc({\tim})
    \end{pmatrix}\! ,
\end{align}
which is again completed by the output
\begin{align}
     & y= 
\begin{pmatrix}
    -\Atop[I] & 0  \\ 0 & 0 \\ 0 & 0 \\ 0 & - I
    \end{pmatrix}^{\tr}
    \begin{pmatrix}
    \npot{} \\  \flux^{-1}(\flux[L]) \\ \charge^{-1}(\charge[C]) \\  \jcurrv[V]
    \end{pmatrix} 
=
     \begin{pmatrix} -\Atop[I]^{\tr} e  \\ -\jcurrv[V] \end{pmatrix}.
 \end{align}
 \end{subequations}
Both models will be shown to fit into the port-Hamiltonian framework which will be presented in the forthcoming section. The first model contains less equations and unknowns, and shares the index analysis results with those for the charge/flux-oriented MNA equations from \cite{Estevez-Schwarz_2000aa} as shown in Section~\ref{sec.index}, whereas the second model is slightly higher structured than the first one.

\section{Port-Hamiltonian formulation of electric circuits}
\label{sec:phs.for.networks}
In this section, we introduce the class of nonlinear port-Hamiltonian DAE systems, for short {\em PH-DAE}, used in this paper. The following system class is a modification of a~class of port-Hamiltonian differential-algebraic equations introduced by \textsc{Mehrmann} and \textsc{Morandin} in \cite{MehM19ppt}. We will show that our circuit models \eqref{model1} and \eqref{model2} fit into this framework.
%
Furthermore, in the second part of this section, we look into multiply coupled PH-DAEs.

\subsection{Port-Hamiltonian for an overall system}
%
%
%
\begin{definition}[Port-Hamiltonian differential-algebraic equation (PH-DAE)]\label{def:PH-network-DAE}
%
%
A diffe\-rential-algebraic equation of the form
\begin{equation}
\label{eq:def.phdae}
\begin{aligned}
\ddt E x(t)
    &= J z(x(t))-r(z(x(t))) +B u(t),
    \\
y(t) & =  B^{\tr} z(x(t))
\end{aligned}
\end{equation}
is called a {\em port-Hamiltonian differential-algebraic equation (PH-DAE)}, if the following holds:
\begin{itemize}
    \item $E\in\real^{k\times n}$, $J\in\real^{n\times n}$ and $B\in\real^{n\times m}$,
    \item $z,r:\real^{n}\to \real^{k}$,
    \item There exists a~subspace $\mathcal{V}\subset\real^n$ with the following properties:
    \begin{itemize}
        \item[(i)] for all intervals $\interval\subset\real$ and functions $u:\interval\to\real^m$ such that \eqref{eq:def.phdae} has a~solution $x:\interval\to\real^n$, it holds $z(x(t))\in\mathcal{V}$ for all $t\in\interval$.
        \item[(ii)] $J$ is skew-symmetric on $\mathcal{V}$. That is, 
        \[\forall v,w\in\mathcal{V}:\;v^\top J w=-w^\top J v.\]
\item[(iii)] $r$ is accretive on $\mathcal{V}$. That is, 
        \[\forall v\in\mathcal{V}:\;v^\top r(v)\geq0.\]
    \end{itemize}
\item There exists some function $H\in C^1(\real^n,\real)$ such that 
\[ \forall x\in z^{-1}(\mathcal{V}):\; \nabla H(x)=E^{\tr}z(x).
\]
\end{itemize}
%
%
%
\end{definition}
Port-Hamiltonian systems an energy balance. In doing so, notice that the total energy of a~PH-DAE at time $t$ is given by $H(x(t))$, whereas the power inflow is realized by the inner product of input and output.
\begin{mylemma}[Energy balance]
The PH-DAE~\eqref{eq:def.phdae} system provides the usual energy balance
\begin{equation}
\ddt H(x(t))  \le  y(t)^\top u(t)\label{eq:energbal}
\end{equation}
of port-Hamiltonian systems.
\begin{proof}
 By using that for any solution $(x,u,y):\interval\to\real^n\times\real^m\times\real^m$ of \eqref{eq:def.phdae}, the following holds: First notice that,
for a~projector $P$ onto $\im E^{\tr}$, we have that $Px:\interval\to\real^n$ is differentiable. Further, by $\nabla H(x)=E^{\tr}z(x)$ for all $x\in z^{-1}(\mathcal{V})$, we have 
\[
\begin{aligned}
\ddt H(x(t))= & (\nabla H(x(t)))^{\tr}  \ddt P{x}(t)
   = z(x(t))^{\tr}E\ddt P{x}(t) 
   \\
    = &
        z(x(t))^{\tr}\ddt E P{x}(t)  
            = z(x(t))^{\tr}\ddt E{x}(t)
 \\
   = &
        \underbrace{z(x(t))^{\tr} J z(x(t))}_{=0} \, \underbrace{-z(x(t))^{\tr}r(z(x(t)))}_{\leq 0} \, +\underbrace{z(x(t))^{\tr}B u(t)}_{=(B^{\tr}z(x(t)))^{\tr}u(t)=y(t)^{\tr}u(t)}
\end{aligned}
\]
Integrating the above expression with respect to time gives for all $t_1\geq t_0$
\[
\begin{aligned}
    H(x(t_1))-H(x(t_0))=        
        & -\int_{t_0}^{t_1}z(x(t))^{\tr}r(z(x(t))) \diff t 
            + \int_{t_0}^{t_1}y(t)^{\tr}u(t) \diff t 
        \\
    \leq& \int_{t_0}^{t_1}y(t)^{\tr}u(t) \diff t.
\end{aligned}\]
This completes the proof.
\end{proof}
\end{mylemma}

\begin{remark}\
\begin{itemize}
 \item[(a)] The function $r$ is responsible for energy dissipation. If $r=0$, then the energy balance \eqref{eq:energbal} becomes an equation. In particular, the energy of the system is conserved, if $r=0$ and $u=0$.
 \item[(b)] 
Our definition of a port-Hamiltonian differential-algebraic equation differs from the one by \textsc{Mehrmann} and \textsc{Morandin} in \cite{MehM19ppt}, which is more general in the sense that time-varying port-Hamiltonian differential-algebraic systems are considered, and the matrices $E$ and $J$ may depend on the state $x$. However, the definition of a differential-algebraic port-Hamiltonian system in \cite{MehM19ppt} does not involve a (possibly proper) subspace $\mathcal{V}\subset\real^n$ on which $z(x(\cdot))$ evolves and the function $r$ is assumed to be linear in $z$. 
 We note, that Definition \ref{eq:def.phdae} can also be extended to the time-varying situation, and to the case of $z$ dependent matrices $E$ and $J$.
 \item[(c)] 
 The space $\mathcal{V}\subset\real^n$ may be proper because of linear (hidden) algebraic constraints. For instance, if for some matrix $K\in\real^{k\times n}$ holds $KE=0$, $KB=0$ and $Kr(z)=0$ for all $z\in\real^n$, then a~multiplication of \eqref{eq:def.phdae} from the left with $K$ leads to 
\[KJz(x(t))=0.\]
This means that the solutions of \eqref{eq:def.phdae} fulfill $z(x(t))\in\ker KJ$ for all $t\in\interval$. \myendofenv
\end{itemize}
\end{remark}

\subsection{Electric Networks---A PH-DAE description}
We show, that the above models~\eqref{model1} and~\eqref{model2} of the electric circuit equations, which are based on the charge/flux-oriented MNA circuit equations, match with the PH-DAE definition.

\begin{proposition}\label{prop:hamiltonian-network}
Let Assumption~\ref{ass:soundness-passivity} hold. Moreover, let $V_C$ and $V_L$ be defined as in Proposition~\ref{prop:energy}. Then the following holds:
\begin{enumerate}
\item[(a)] The model \eqref{model1} is a PH-DAE with
\[ 
\begin{aligned}
 u(t)&= \begin{pmatrix}
    \csrc({\tim}) \\ \vsrc({\tim})
    \end{pmatrix},\quad y(t)=
     \begin{pmatrix} -\Atop[I]^{\tr} e(t)  \\ -\jcurrv[V](t) \end{pmatrix},\quad
     x(t)=
   \begin{pmatrix}
       \charge[C](t) \\ \flux[L](t) \\ \npot{}(t) \\  \jcurrv[V](t)
    \end{pmatrix},\\
z(x) &= 
    \begin{pmatrix}
        \npot{} \\  \jcurrv[L] \\ u_C \\  \jcurrv[V]
    \end{pmatrix}
        = \begin{pmatrix}
    \npot{} \\  \flux^{-1}(\flux[L]) \\ \charge^{-1}(\charge[C]) \\   \jcurrv[V]
    \end{pmatrix},\quad  
r\left(\begin{pmatrix}
        \npot{} \\  \jcurrv[L] \\ u_C \\  \jcurrv[V]
    \end{pmatrix}\right)=\begin{pmatrix}
    \Atop[R]g(\Atop[R]^{\tr}\npot{})\\0\\ 
        \Atop[C]^\tr \npot - u_{C} \\0
    \end{pmatrix},\\ 
    E&=\begin{pmatrix}
     \Atop[C] & 0 & 0 & 0 \\
        0 & I & 0 & 0 \\
        0 & 0 & 0 & 0 \\
        0 & 0 & 0 & 0 \\
    \end{pmatrix},\quad J=
    \begin{pmatrix}
         0 & -\Atop[L] & 0 & -\Atop[V] \\
         \Atop[L]^{\tr} & 0 & 0 & 0 \\
         0 & 0 & 0 & 0 \\
        \Atop[V]^{\tr} & 0 & 0 & 0 \\
    \end{pmatrix}\quad B=\begin{pmatrix}
    -\Atop[I] & 0 \\ 0 & 0 \\ 0 & 0 \\ 0 & - I
    \end{pmatrix},
\end{aligned}\]
subspace 
\[\mathcal{V}=\left\{\left.\begin{pmatrix}
        \npot{} \\  \jcurrv[L] \\ u_C \\  \jcurrv[V]
    \end{pmatrix}\in\real^n\right|\Atop[C]^{\tr}e=u_C\right\}.\]
and Hamiltonian
\[
H(x)= V_C(\charge[C]) + V_L(\flux[L]).
\]
\item[(b)] The model \eqref{model2} is a PH-DAE with $u(t)$, $y(t)$ as in (a), and
\[ 
\begin{aligned}
x(t)&=
   \begin{pmatrix}
       \npot{}(t) \\ \jcurrv[C](t) \\ \charge[C](t) \\ \flux[L](t) \\  \jcurrv[V](t)
    \end{pmatrix},\;
z(x) = 
    \begin{pmatrix}
        \npot{} \\  \jcurrv[C] \\ u_C \\\jcurrv[L] \\  \jcurrv[V]
    \end{pmatrix}
        = \begin{pmatrix}
    \npot{} \\  \jcurrv[C] \\ \charge^{-1}(\charge[C]) \\ \flux^{-1}(\flux[L]) \\   \jcurrv[V]
    \end{pmatrix},\\
  r\left(\begin{pmatrix}
        \npot{} \\  \jcurrv[C] \\ u_C \\\jcurrv[L] \\  \jcurrv[V]
    \end{pmatrix}\right)&=\begin{pmatrix}
    \Atop[R]g(\Atop[R]^{\tr}\npot{})\\0\\ 
        0 \\0 \\0
    \end{pmatrix},\quad 
E=\begin{pmatrix}
     0 & 0 & 0 & 0& 0 \\
     0 & 0 & 0 & 0& 0 \\
     0 & 0 & I & 0 & 0 \\
     0 & 0 & 0 & I & 0 \\
     0 & 0 & 0 & 0 & 0 \\
    \end{pmatrix},\\
 J&=\begin{pmatrix}
         0 &  -\Atop[C]&0  &-\Atop[L]   & -\Atop[V] \\
        \Atop[C]^{\tr} & 0 & -I & 0 & 0 \\
        0 & I & 0 & 0 & 0 \\
        \Atop[L]^{\tr} & 0 & 0 & 0 & 0 \\
        \Atop[V]^{\tr} & 0 & 0 & 0 & 0 \\
    \end{pmatrix},\; B=\begin{pmatrix}
    -\Atop[I] & 0 \\ 0 & 0 \\ 0 & 0 \\ 0 & 0 \\ 0 & - I
    \end{pmatrix},
\end{aligned}\]
and, for $n_v$ being the number on non-grounded vertices, subspace 
\[\mathcal{V}=\real^{n_v}\times\real^{n_C}\times\real^{n_C}\times\real^{n_L}\times\real^{n_V},
\]
and Hamiltonian
\[
H(x)= V_C(\charge[C]) + V_L(\flux[L]).
\]
\end{enumerate}
\end{proposition}

\begin{proof}
\begin{enumerate}
\item[(a)] Since \eqref{model1} contains the equation $\Atop[C]^{\tr}e(t)-u_C(t)=0$, we see that any solution fulfills $z(x(t))\in\mathcal{V}$ pointwise. The skew-symmetry of $J$ is obvious. Further, by the assumption that the Jacobian of $g$ has positive definite real part, we obtain that $g$ is accretive. This directly implies that $r$ is accretive on $\mathcal{V}$. Moreover, by using Proposition~\ref{prop:energy}, we compute
\[\begin{aligned}
\nabla H(x)=&\begin{pmatrix}\nabla V(\charge[C])\\\nabla V(\flux[L])\\0\\0\end{pmatrix}
\stackrel{\text{Prop.~\ref{prop:energy}}}{=}\begin{pmatrix}\charge^{-1}(\charge[C])\\
\flux^{-1}(\flux[L])\\0\\0
\end{pmatrix}
=\begin{pmatrix}u_C\\
\jcurrv[L]\\0\\0
\end{pmatrix}\\
\stackrel{z\in\mathcal{V}}{=}&\begin{pmatrix}\Atop[C]^{\tr}e\\
\jcurrv[L]^{\tr}\\0\\0
\end{pmatrix}
= \begin{pmatrix}\Atop[C] & 0 & 0 & 0 \\
        0 & I & 0 & 0 \\
        0 & 0 & 0 & 0 \\
        0 & 0 & 0 & 0
    \end{pmatrix}^{\tr}\begin{pmatrix}
        \npot{} \\  \jcurrv[L] \\ u_C \\  \jcurrv[V]
    \end{pmatrix}=E^{\tr}z(x).
\end{aligned}\]
\item[(b)] The space $\mathcal{V}=\real^{n_v}\times\real^{n_C}\times\real^{n_C}\times\real^{n_L}\times\real^{n_V}$ trivially has the property that all solutions evolve in $\mathcal{V}$. Moreover, $J$ is skew-symmetric, and the accretivity of $r$ follows from the accretivity of $g$, where the latter can be concluded by the argumentation as in (a). For the gradient of the Hamiltonian, we compute
\[\begin{aligned}
\nabla H(x)=&\begin{pmatrix}0\\0\\\nabla V(\charge[C])\\\nabla V(\flux[L])\\0\end{pmatrix}
\stackrel{\text{Prop.~\ref{prop:energy}}}{=}\begin{pmatrix}0\\0\\\charge^{-1}(\charge[C])\\
\flux^{-1}(\flux[L])\\0
\end{pmatrix}
=\begin{pmatrix}0\\0\\u_C\\
\jcurrv[L]\\0
\end{pmatrix}\\
=& \begin{pmatrix}0 & 0 & 0 & 0 & 0 \\
0 & 0 & 0 & 0 & 0 \\
0 & 0 & I & 0 & 0 \\
0 & 0 & 0 & I & 0 \\
0 & 0 & 0 & 0 & 0 
    \end{pmatrix}^{\tr}\begin{pmatrix}
        \npot{} \\  \jcurrv[C] \\ u_C \\  \jcurrv[V] \\ \jcurrv[V]
    \end{pmatrix}=E^{\tr}z(x)
\end{aligned},\] 
which concludes the proof.\qedhere
\end{enumerate}
\end{proof}

\subsection{Port-Hamiltonian system formulation  for multiple  subsystems}
%
In the following, we generalize the above monolithic setting of Definition \ref{def:PH-network-DAE} to the case of $\numsys\ge 2$ subsystems. 
%
To couple several PH-DAEs,
we first setup some notation, to address different types of input and output: internal and coupling quantities.

\begin{definition}[Multiply coupled PH-DAE]\label{def:multi-PH-network-DAE}
We consider $\numsys$ copies of PH-DAEs~\eqref{eq:def.phdae} 
\begin{equation}\label{phdae.coupled.single.network}
\begin{aligned}
\ddt E_i  x_i(t) 
    =& J_i  z_i(x_i(t)) -r_i\bigl(z_i (x_i(t))\bigr) + B_i u_i (t)
    \\
y_i(t) = &   B_i^\tr z_i \bigl(x_i(t) \bigr)
 \end{aligned}
\end{equation}
with associated Hamiltonian $H_i$
($i=1,\ldots,\numsys$). 
We call these $k$ copies of PH-DAEs a \emph{multiply coupled PH-DAE} if the follwing are satisfied: The input $u_i$ and the output $y_i$ are split into
\begin{equation}
    \label{eq:split-input-output}
    u_i(t)= \begin{pmatrix}
                \hat u_i(t) \\
                \bar u_i(t)
        \end{pmatrix},
        \quad
    y_i(t)= \begin{pmatrix}
                \hat y_i(t) \\
                \bar y_i(t)
        \end{pmatrix},
\end{equation}
where the bar-accent refers to external inputs and outputs, i.e., quantities, which are not communicated to other subsystems, and the hat-accented quantities refer to input and output data used for coupling of the $\numsys$ subsystems.
Moreover, the port matrix is split accordingly:
\begin{equation}\label{eq:split-port-feed-matrices}
B_i = \begin{pmatrix}
                \hat B_i &
                \bar B_i
            \end{pmatrix} .
\end{equation}
The subsystems are coupled via 
topological coupling matrices
$\hat{C}_{i,j} \in \{-1,\, 0,\, 1 \}^{m_i\times m_i}$ 
\[
\hat{u}_i + \sum_{j=1, j \neq i}^k \hat C_{i,j} \hat{y}_j =0 \qquad (\text{for }\; i=1,\dotsc,\numsys), \quad  \color{black}
\hat{C}=
    \begin{pmatrix}
       0             &  \hat{C}_{1,2} & \dots & \hat{C}_{1,\numsys} \\
       \hat{C}_{2,1} &  \ddots        &  \ddots & \vdots \\
       \vdots        &\ddots          &  \ddots & \hat{C}_{\numsys-1,\numsys}\\
        \hat{C}_{\numsys,1} & \dots & \hat{C}_{\numsys,\numsys-1} & 0
\end{pmatrix}
\]
\color{black}
with $\hat{C}$ skew symmetric.
\end{definition}
\color{black}
%

Now, we can deduce for the overall system described in Definition \ref{def:multi-PH-network-DAE}:

\begin{corollary}[Multiply skew-symmetric coupling structure preserving interconnection]
\label{cor.structure.preserving.interconnection.network}
We consider a multiply coupled PH-DAE  with $\numsys$ subsystems. 
The overall system is obtained by aggregation of vector quantities and matrices:
\begin{alignat*}{2}
& v^{\tr} =  (v_1^{\tr},\ldots, v_\numsys^{\tr})
    \quad  &&\text{ for } v \in \{x,\, u,\, \hat u, \,\bar u, \, y, \, \hat y, \, \bar y\},
        \\
 & F=\diag\,(F_1,\ldots, F_\numsys)
    \quad &&\text{ for } F \in \{E,\,J,\, \hat{B},\, \bar{B}\},
\end{alignat*}
\begin{equation*}
    r^{\tr}(z(x)) = \left(r_1 \bigl(z_1(x_1)\bigr)^{\!\tr},\, \dotsc,\,
    r_{\numsys} \bigl(z_\numsys(x_\numsys)\bigr)^{\tr}  \right), 
    \;\quad
    z^{\tr}(x) = \left((z_1(x_1)^{\!\tr},\, \dotsc,\,
                     z_\numsys(x_\numsys)^{\tr}  \right), 
\end{equation*}
%
and it reads (with coupling equation $\hat{u}+\hat{C}\hat{y}=0$
 in the third block equation)
%
\begin{equation}
\label{eq:phdae.coupled.joint.network}
\begin{aligned}
\ddt 
\begin{pmatrix}
E & 0 & 0 \\
0 & 0 & 0 \\
0 & 0 & 0
\end{pmatrix}\!\!
\begin{pmatrix}
 x \\ {\hat u} \\ { \hat y}
\end{pmatrix}
 = &
\begin{pmatrix}
J & \hat{B} & 0 \\
- \hat{B}^{\tr} & 0 & I \\
0 & -I &  -\hat{C}
\end{pmatrix}\!\!
\begin{pmatrix}
 z \\  \hat u \\  \hat y
\end{pmatrix}
\!-\!
\begin{pmatrix}
r \\ 0 \\ 0
\end{pmatrix}
+ \begin{pmatrix}
\bar B  \\ 0 \\ 0
\end{pmatrix}
\bar{u},
\\
\bar y  = &
\begin{pmatrix}
\bar{B}^{\tr} &0 &0
\end{pmatrix}
\begin{pmatrix}
z \\ \hat{u} \\ \hat{y}
\end{pmatrix} \!. 
\end{aligned}
\end{equation}
\color{black}
Then this system is a PH-DAE with Hamiltonian $H=H_1 + \dotsc + H_{\numsys}$.
\end{corollary}
%
\begin{proof}
In order to simply superpose the subsystems,
we rewrite the $i$th subsystem~\eqref{phdae.coupled.single.network} in a matrix format. 
To this end, we use split input and output: both comprise coupling terms and external terms. 
Thereby, the coupling terms will belong to the internal description of the overall systems. Only external input/output will form the input/output of the overall systems. Subsystem~\eqref{phdae.coupled.single.network} can be equivalently written as 
\begin{equation}\label{eq:ext-ph-DAE-representation}
\begin{aligned}
    \ddt\begin{pmatrix}
        E_i & 0  \\
        0 & 0
    \end{pmatrix}\!\!
\begin{pmatrix}
x_i \\ {\hat u}_i
\end{pmatrix}
= &
\begin{pmatrix}
J_i & \hat{B}_i  \\
-\hat{B}_i^{\tr} &  0
\end{pmatrix}\!\!
\begin{pmatrix}
 z_i(x_i) \\  \hat u_i
\end{pmatrix}
-
\begin{pmatrix}
r_i\left( z_i(x_i) \right) \\ 0
\end{pmatrix}
+ 
\begin{pmatrix}
\bar{B}_i   & 0 \\0 & I %
\end{pmatrix}\!\!
\begin{pmatrix}
\bar u_i \\ \hat y_i
\end{pmatrix}\!\!
\\
\begin{pmatrix}
\bar y_i \\ \hat{d}_i
\end{pmatrix} \!\!= &
\begin{pmatrix}
\bar B_i^{\tr}\! &0 \\ 0 & I %
\end{pmatrix}\!\!
\begin{pmatrix}
z_i(x_i) \\ \hat{u}_i
\end{pmatrix} \!,
\end{aligned}
\end{equation}
\color{black}
where we use the additional dummy output $\hat{d}_i=\hat{u}_i$.
%
Then, the extended system~\eqref{eq:ext-ph-DAE-representation} is again a PH-DAE,
with corresponding extended matrices: 
\begin{gather*}
\tilde E_i := \begin{pmatrix}
E_i & 0\\
0 &  0
\end{pmatrix},\quad
\tilde J_i := \begin{pmatrix}
J_i & \hat{B}_i \\
-\hat{B}_i^{\tr} & 0
\end{pmatrix}, \quad
\tilde B_i:= \begin{pmatrix}
\bar B_i & 0 \\ 0 & I 
\end{pmatrix}.
\end{gather*}
Now, we discuss every block of equations in the joint system~\eqref{eq:phdae.coupled.joint.network}. First, the aggregation
$F=\diag\,(F_1,\ldots, F_\numsys)$ for $F \in \{E,\,J,\, \hat{B},\, \bar{B}\}$
\color{black}
of
(\ref{eq:ext-ph-DAE-representation}.1) yields directly (\ref{eq:phdae.coupled.joint.network}.1) padded with zeros for the variable $\hat{y}$. 
For the second block of equations, we have to perform aggregation and have to move $\hat{y}$ from the output position to internal variables.
Thereby the vector $(x^\tr, \hat{u}^\tr)$ and $(z^\tr, \hat{u}^\tr)$ are extended. 
Then,
the aggregated structure preserving interconnection $\hat{u}=- \hat{C} \hat{y}$ gives the third block. Finally, the output equation of \eqref{eq:ext-ph-DAE-representation} yields the output equation by aggregation, dropping the dummy part and 
adding a padding of zeros. The properties of the terms are inherited from the respective definition of the subsystems. 
%
\end{proof}

\begin{remark}
This transfers the result from~\cite{MehM19ppt} to  circuits with non-linearities. 
Furthermore, no additional variables are introduced. Moreover the structure matrix of the overall 
system~\eqref{eq:phdae.coupled.joint.network} is identified as 
\begin{equation}
    J^{\tot}:= \begin{pmatrix}
            J & \hat{B} & 0 \\
            -\hat{B}^{\tr} & 0 & I \\
            0 & -I &  -\hat{C}
        \end{pmatrix}\! . \tag*{$\myendofenv$}
\end{equation}
\end{remark}
\begin{remark}
\label{rem.change.ham.coupled.network}
\begin{enumerate}[label=\roman*)]
\item System ~\eqref{eq:phdae.coupled.joint.network} can be condensed to a PH-DAE 
(by removing internal input $\hat{u}$ and output $\hat{y}$)
\begin{subequations}
\label{eq:phdae.coupled.joint.network.condensed}
\begin{align}
\ddt E  x& =  \hat J z  - r +  \bar B \bar u, \\
\bar y & =  \bar B^{\tr} z
\end{align}
\end{subequations}
with the skew-symmetric matrix $\hat J$ given by  $\hat J = J - \hat B \hat C \hat B^{\tr}$. This follows directly from
$J z + \hat B \hat u = J z - \hat B \hat C \hat y = (J- \hat B \hat C \hat B^{\tr}) z$.
Thereby the PH-DAE structure is kept.
%

\item Note that the change in the Hamiltonian $H$ of~\eqref{eq:phdae.coupled.joint.network}, as well as in its condensed
version~\eqref{eq:phdae.coupled.joint.network.condensed}, from time $t$ to $t+h$ is given by
\begin{align}
\label{change.hamiltonian.condensed}
& \int_t^{t+h} 
- z(x(\tau))^{\tr} r(z(x(\tau))) +
\bar u(\tau)^{\tr} \bar y(\tau)
\, \diff \tau  \\
& 
\qquad =
\int_t^{t+h} 
- z(x(\tau))^{\tr} r(z(x(\tau))) +
\color{black}
\bar u(\tau)^{\tr}
\bar B^{\tr}
z(x(\tau))\color{black} \, \diff \tau. 
\tag*{$\myendofenv$}
\end{align}
%
\end{enumerate}
\end{remark}

\subsection{Electric circuits with multiple subsystems---A PH-DAE 
description}
\label{sec3.4}

Large integrated circuits are usually designed in blocks which may comprise even different  functional units. Then, these subcircuits are put together in an overall system by connecting respective terminals. In this way, a substructure may be already given by the circuit design, see e.g.~Figure \ref{fig.macro} (left) 
with respective inputs $\bar{u}$ and outputs $\bar{y}$. 
To form separate models of the subcircuits, one can artificially double the vertices of the subsystems' terminals by inserting a voltage source which provides a voltage drop of zero (artificial voltage source). 
This amounts to further inputs and outputs for the subsystems, which state the coupling  $\hat{u}$ and $\hat{y}$, 
see Figure \ref{fig.macro} (right).

Let the overall circuit (with given Assumption \ref{ass:soundness-passivity}), consist of subcircuits $i=1,\ldots,\numsys$. We use the index $i$ to identify the quantities of the $i$th subcircuit, e.g.~we use $\npot[i](t)\in\real^{n_{u_i}}$ for the vertex voltages and so on.
Moreover, we assume that we have $n_{\lambda}$ coupling edges linking the $\numsys$ subcircuits in the overall setting. Then we have  associated edge
currents $\lambdav(t) \in \real^{n_{\lambda}}$ and  $n_{\lambda}$ artificial voltage source. Now, let the $i$th subsystem have the respective incidence matrix $\Atop[\lambda_{\mathnormal i}] \in \{ -1,\, 0,\, 1 \}^{n_{u_i}\times n_{\lambda}}$ for the artificial voltage sources. Thus the coupling amounts to (i) an additional term in the KCL ($i$th circuit), for the coupling edge/current: $\Atop[\lambda_{\mathnormal i}] \lambdav$. In fact, one can model this by adding this contribution to the current source term ($\Atop[I]$):
\[
   \Atop[I_{\mathnormal i}] \rightsquigarrow \left( \Atop[I_{\mathnormal i}],\, \Atop[\lambda_{\mathnormal i}] \right),
        \quad \csrc[i] \rightsquigarrow \begin{pmatrix} \csrc[i] \\ \lambdav \end{pmatrix} .
\]
Due to the virtuality of the coupling voltage sources, one has (ii) to guarantee that the vertex potentials at the boundaries coincide, as done in \eqref{CouplingCondCircuits}, see below.

\begin{figure}
        \centering
    \scalebox{0.7}{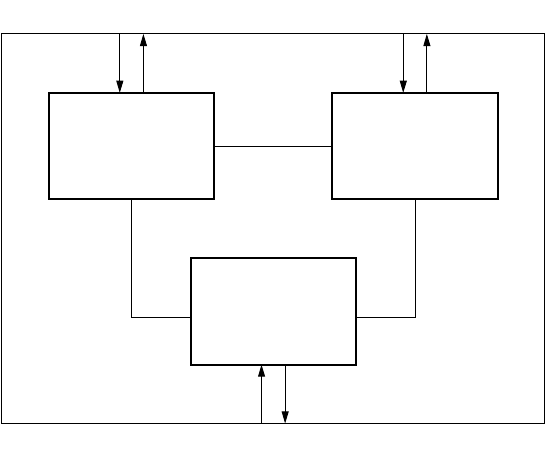} \hfill
    \scalebox{0.7}{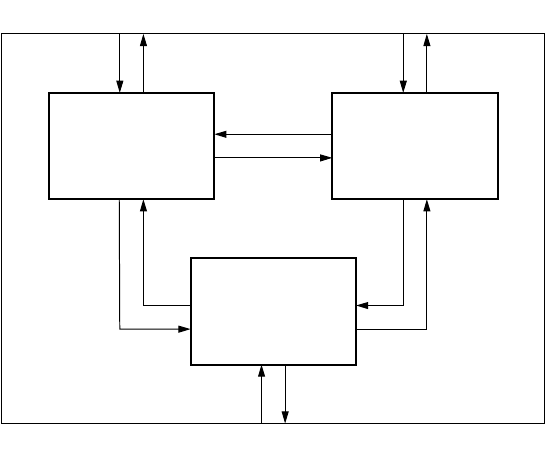}

        \caption{\label{fig.macro}
         Input/output for distributed circuits: monolithic view (left), coupled circuits view (right).}
\end{figure}

In the end, the circuit equations for the $\numsys$ coupled circuit DAEs are comprised by the subsystems $i=1,\dotsc,\numsys$:
\begin{subequations}\label{CoupledCircuits}
\begin{align}\label{SubSystemCircuitsKCL}
%
 \vect{0} & =
 \Atop[C_{\mathnormal i}] \ddt \charge[C_{\mathnormal i}](\Atop[C_{\mathnormal i}]^{\tr} \npot[i])
    + \Atop[R_{\mathnormal i}]  \resi_i ( \Atop[R_{\mathnormal i}]^{\tr} \npot[i])
    + \Atop[L_{\mathnormal i}] \jcurrv[L_{\mathnormal i}]
    + \Atop[V_{\mathnormal i}] \jcurrv[V_{\mathnormal i}] 
%
    + \Atop[I_{\mathnormal i}] \csrc[i] (t)
    + \Atop[\lambda_{\mathnormal i}] \lambdav
 \\
 											\label{SubSystemCircuitsMitte}
 \vect{0} & =  \ddt \flux[L_{\mathnormal i}](\jcurrv[L_{\mathnormal i}])
    - \Atop[L_{\mathnormal i}]^{\tr} \npot[i]
 \\
 \vect{0} & =  \Atop[V_{\mathnormal i}]^{\tr} \npot[i] - \vsrc[i](t)
 %
  \label{SubSystemCircuitsIdFlux}
\end{align}
together with a set of linear coupling equations
\begin{equation}\label{CouplingCondCircuits}
  \vect{0} = \sum_{i=1}^\numsys \Atop[\lambda_{\mathnormal i}]^{\tr} \,\npot[i] \, .
\end{equation}
\end{subequations}

These coupled DAE circuit equations can be written as $\numsys$ multiply coupled PH-DAE system according to Definition~\ref{def:multi-PH-network-DAE}.
The only ambiguity is the handling of the coupling condition~\eqref{CouplingCondCircuits}. The simplest approach is to add the coupling condition to one of the subsystem, without loss of generality to the last one. It holds

\begin{mylemma}[PH-DAE formulation of mutually coupled DAEs]
\label{lem:PH-DAE_mutally_coupled_networks}
The coupled circuit DAEs~\eqref{CoupledCircuits} define $\numsys$ multiply coupled PH-DAE systems according to Definition~\ref{def:multi-PH-network-DAE}.
\end{mylemma}

\begin{proof}
For  $i=1,\ldots,\numsys-1$ we set 
\begin{subequations}
\label{basis.dynit.network.1tok-1}
\begin{align}
x_i &= \begin{pmatrix}
    {\charge[]}_{C_i} \\ {\flux[]}_{L_i} \\ {\npot{}}_i \\ {\jcurrv[V_i]}    \end{pmatrix}, \quad 
z_i = \begin{pmatrix}
    {\npot{}}_i \\  {\jcurrv[L_i]} \\ {u}_{C_i} \\  {\jcurrv[V_i]}
    \end{pmatrix}, \quad
    \bar u_i = \begin{pmatrix}
    \csrc[i]({\tim}) \\ \vsrc[i]({\tim})
    \end{pmatrix}, \quad
    \hat{u}_i+\hat y_k=0, \\
E_i &= \begin{pmatrix}
     \Atop[C_i] & 0 & 0 & 0 \\
        0 & I & 0 & 0 \\
        0 & 0 & 0 & 0 \\
        0 & 0 & 0 & 0 \\
    \end{pmatrix}, \quad 
    J_i =  \begin{pmatrix}
         0 & -\Atop[L_i] & 0 & -\Atop[V_i] \\
         \Atop[L_i]^{\tr} & 0 & 0 & 0 \\
         0 & 0 & 0 & 0 \\
        \Atop[V_i]^{\tr} & 0 & 0 & 0 \\
    \end{pmatrix}, \quad
 \\
    r_i &= \begin{pmatrix}
    \Atop[R_i]g_i(\Atop[R_i]^{\tr}{\npot[]}_i)\\0\\ 
        \Atop[C_i]^\tr {\npot[]}_i - u_{C_i} \\0
    \end{pmatrix}, \quad
    \hat B_i = \begin{pmatrix}
     \Atop[\lambda_i] \\ 0 \\ 0 \\ 0
    \end{pmatrix}, \quad
    \bar B_i = \begin{pmatrix} 
    -\Atop[I_i] & 0  \\ 0 & 0 \\ 0 & 0 \\ 0 & - I
    \end{pmatrix},
\end{align}
\end{subequations}
and for $i=\numsys$ the definition
\begin{subequations}
\label{basis.dynit.network.k}
\begin{align}
x_k &= \begin{pmatrix}
    {\charge[]}_{C_k} \\ {\flux[]}_{L_k} \\ {\npot{}}_k \\ {\jcurrv[V_k]} \\ \lambdav   \end{pmatrix}, \quad 
z_k = \begin{pmatrix}
    {\npot{}}_k \\  {\jcurrv[L_k]} \\ {u}_{C_k} \\  {\jcurrv[V_k]} \\ \lambdav
    \end{pmatrix}, \quad
    \bar u_k  = \begin{pmatrix}
    \csrc[k]({\tim}) \\ \vsrc[k]({\tim})
    \end{pmatrix}, \quad
    \hat{u}_k- \sum_{i=1}^{\numsys-1}\hat y_i=0, \\
E_k &= \begin{pmatrix}
     \Atop[C_k] & 0 & 0 & 0 & 0\\
        0 & I & 0 & 0 & 0 \\
        0 & 0 & 0 & 0 & 0\\
        0 & 0 & 0 & 0 & 0\\
        0 & 0 & 0 & 0 & 0
    \end{pmatrix}, \quad 
    J_k =  \begin{pmatrix}
         0 & -\Atop[L_k] & 0 & -\Atop[V_k] & -\Atop[\lambdav_k] \\
         \Atop[L_k]^{\tr} & 0 & 0 & 0 & 0\\
         0 & 0 & 0 & 0 & 0\\
        \Atop[V_k]^{\tr} & 0 & 0 & 0 & 0\\
        \Atop[\lambdav_k]^{\tr} & 0 & 0 & 0 &0
    \end{pmatrix}, \quad
 \\
    r_k &= \begin{pmatrix}
    \Atop[R_k]g_k(\Atop[R_k]^{\tr}{\npot[]}_k)\\0\\ 
        \Atop[C_k]^\tr {\npot[]}_k - u_{C_k} \\0\\0
    \end{pmatrix}, \quad
    \hat B_k = \begin{pmatrix}
     0 \\ 0 \\ 0 \\ 0 \\ I
    \end{pmatrix}, \quad
    \bar B_k = \begin{pmatrix} 
    -\Atop[I_k] & 0  \\ 0 & 0 \\ 0 & 0 \\ 0 & - I \\ 0 & 0
    \end{pmatrix},
\end{align}
\end{subequations}
completes the proof.
\color{black}
\end{proof}
In addition, the joint system has a PH-DAE formulation, too.
\begin{mylemma}[PH-DAE formulation of coupled circuit DAEs\color{black}]
\label{remark.network.coupled.jpint}
The coupled circuit equations~\eqref{CoupledCircuits}, written as a single system, 
can be represented as 
PH-DAE in the condensed form~\eqref{eq:phdae.coupled.joint.network.condensed}.
\begin{proof}
Here we set
\begin{subequations}
\label{eq.network.coupled.joint}
\begin{align}
    x&:= \begin{pmatrix} \charge[C] \\ \flux[L] \\ \npot{} \\ \jcurrv[V] \\ \lambdav \end{pmatrix}\!, 
        \quad
     z(x) := 
    \begin{pmatrix}
        \npot{} \\  \jcurrv[C] \\ u_C \\\jcurrv[L] \\  \jcurrv[V]
    \end{pmatrix}
        = \begin{pmatrix}
    \npot{} \\  \jcurrv[C] \\ \charge^{-1}(\charge[C]) \\ \flux^{-1}(\flux[L]) \\   \jcurrv[V]
    \end{pmatrix}\!, \\
    E&:=  \begin{pmatrix} \Atop[\Cnet] & 0 & 0& 0 & 0 \\
            0 & I & 0 & 0 & 0\\
            0 & 0 & 0 & 0 & 0\\
            0 & 0 & 0 & 0 & 0 \\
            0 & 0 & 0 & 0 & 0 
    \end{pmatrix}\!,
    \quad
    r:=  \begin{pmatrix}
    \Atop[R] \resi[R](\Atop[R]^{\tr} \npot{},\tim) \\
        0 \\
        \Atop[C]^{\tr} \npot[] - u_C \\
        0 \\
        0 
    \end{pmatrix}\!, \\
    \hat J&:=  \begin{pmatrix}
            0 & - \Atop[L] & 0 & - \Atop[V] & - \Atop[\lambda]   \\
    \Atop[L]^{\tr} & 0 & 0 & 0 & 0\\
    0 & 0 & 0 & 0 & 0 \\
    \Atop[V]^{\tr} & 0 & 0 & 0 & 0 \\
    \Atop[\lambda]^{\tr} & 0 & 0 & 0 & 0 
    \end{pmatrix}\!, \quad
        \bar B:= \begin{pmatrix}
            -\Atop[I] & 0 \\
            0 & 0 \\
            0 & 0 \\
            0 & - I \\
            0 & 0
    \end{pmatrix}\! ,
    \quad 
    \bar u= \begin{pmatrix} \csrc[] \\ \vsrc[] \end{pmatrix} \!,
    \end{align}
\end{subequations}
where we have used aggregrated matrices
\begin{gather*}
 \Atop[R] g =  (\Atop[R_1]
            g_1(\Atop[R_1]^\tr x_1,t),\, \dotsc,\,
            \Atop[R_\numsys] g_\numsys(\Atop[R_\numsys] x_\numsys,t) )^\tr, \quad \Atop[\lambda]^\tr = (\Atop[\lambda_1]^\tr,\, \dotsc,\, \Atop[\lambda_\numsys]^\tr ), \\
        \Atop[P]:=\diag(\Atop[P_1],\, \dotsc, \, \Atop[P_k])
    \quad \mbox{ for } P \in \{C,R,L,V\}
\end{gather*}
and aggregated quantities
\[
w = \begin{pmatrix}
w_1 \\ \vdots \\ w_\numsys
\end{pmatrix}
\quad
\mbox{ for} \quad
w \in \{\charge[C],\,\flux[L],\,\npot,\,u_C,\,\jcurrv[V],\,\jcurrv[C],\,\jcurrv[L]\}.
\]
The Hamiltonian is given as in Proposition~
\ref{prop:hamiltonian-network}
as the sum of the Hamiltonians of the $\numsys$ subsystems.
\end{proof}
\end{mylemma}
\color{black}


\color{black}

\section{Index analysis of circuit equations}\label{sec:IndAna}
\label{sec.index}
In the field of DAEs, there exist several index concepts, which quantify the distance to the case of ODEs. This can be done with respect to derivatives needed to transform a DAE into an ODE, i.e., the differentiation index~\cite{Hairer_2002}. On the other hand, the perturbation index~\cite{Hairer_2002} quantifies the distance of the solutions to a perturbed system, with respect to the number of derivatives of the perturbation (which may enter the solution). A third concept is the tractability index  \cite{Griepentrog_1986,LamoMarz13}, which is based on a matrix change and reveals the respective components with the minimal regularity required.
In this work, we focus on the differentiation index, which we refer to as {\em index} throughout this article.
\begin{definition}[Derivative array, differentiation index, {\cite[Def.~3,72]{LamoMarz13}}]\label{def:index}
Let $U,V\subset\real^n$ be open and $\interval\subset\real$ be an interval. Let $\nu\in\mathbbm{N}$, $\mathcal{F}:U\times V\times I\to\real^k$, and a~DAE 
\begin{equation}
 \mathcal{F}(\dot{x}(t),x(t),t)=0\label{eq:dae}
\end{equation}
be given. Then the {\em $\nu$th derivative array of \eqref{eq:dae}} is given by the first $\nu$ formal derivatives of \eqref{eq:dae} with respect to time, that is
\begin{equation}
\mathcal{F}_{\nu}(x^{(\nu+1)}(t),x^{(\nu)}(t),\ldots,\dot{x}(t),x(t),t)
    =\begin{pmatrix}
        \mathcal{F}(\dot{x}(t),x(t),t)\\
        \ddt \mathcal{F}(\dot{x}(t),x(t),t)\\
        \vdots\\
        \frac{\diff^\nu}{\diff t^\nu}   
            \mathcal{F}(\dot{x}(t),x(t),t)
      \end{pmatrix}=0.
\end{equation}
The DAE \eqref{eq:dae} is said to have {\em (differentiation) index $\nu\in\mathbbm{N}$}, if for all $(x,t)\in V\times I$, there exists some unique $\dot{x}\in U$ such that there exist some ${x}^{(2)},\ldots,x^{(\nu+1)}\in U$ such that $\mathcal{F}_{\nu}(x^{(\nu+1)},x^{(\nu)},\ldots,\dot{x},x(t),t)=0$. In this case, there exists some function $f:V\times I\to V$ with $(x,t)\mapsto \dot{x}$ for $t$, $x$ and $\dot{x}$ with the above properties. The ODE
\begin{equation}
\dot{x}(t)=f(x(t),t)\label{eq:inhode}
\end{equation}
is said to be {\em inherent ODE of \eqref{eq:dae}}.
\end{definition}

Next we characterize the index of the circuit equations~\eqref{eq:mna.cond_a} and \eqref{eq:mna.cond_b} by means of the properties of the subgraphs corresponding to specific electric components.

\begin{theorem}\label{thm:index-1-network-conditions}
Assumption~\ref{ass:soundness-passivity} shall hold.
\begin{itemize}
    \item The index $\nu$ of the circuit DAE~\eqref{eq:mna.cond_a} fulfills: $\nu=1$  if, and only if, it neither contains cycles only consisting of edges to capacitances \underline{and} voltage sources nor cuts only consisting of edges to inductances \underline{and/or} current sources. Otherwise, $\nu=2$.
\item The index $\nu$ of the circuit DAE~\eqref{eq:mna.cond_b} fulfills: $\nu=1$  if, and only if, it neither contains cycles only consisting of edges to capacitances \underline{and/or} voltage sources nor cuts only consisting of edges to inductances \underline{and/or} current sources. Otherwise, $\nu=2$.
\end{itemize}
 \end{theorem}
\begin{remark}\label{rem:CVcycles}
\begin{itemize}
    \item[(a)]\label{rem:CVcycles_a} There is a~small but nice difference between the indices of DAEs~\eqref{eq:mna.cond_a} and \eqref{eq:mna.cond_b}: Whereas cycles only consisting of edges of capacitances lead to an index  $\nu=2$ of \eqref{eq:mna.cond_b}, this is not necessarily the case for the DAE \eqref{eq:mna.cond_a}. Since cycles only consisting of voltage sources are excluded beforehand by Assumption~\ref{ass:soundness-passivity}, the absence of cycles only consisting of edges to capacitances \underline{and} voltage sources is equivalent to the property of a~circuit that it does not contain any cycles consisting of capacitances \underline{and/or} voltage sources \underline{except for} cycles consisting of capacitances. The latter is, by Proposition~\ref{graphalg}, equivalent to
\begin{equation}\{\jcurrv[V]\in\real^{n_V}\,|\, \Atop[V]\jcurrv[V]\in\im \Atop[C]\}=\{0\}.\label{eq:CVcyc1}\end{equation} 
Now consider a~matrix $Z_C$ with full column rank and $\im Z_C=\ker\Atop[C]^{\tr}$. Then, by taking the orthogonal complement, we obtain $\ker Z_C^{\tr}=\im \Atop[C]$, and a~combination with \eqref{eq:CVcyc1} leads to the fact that a~circuit fulfilling Assumption~\ref{ass:soundness-passivity} does not contain any cycles consisting of capacitances \underline{and} voltage sources if, and only if, 
\begin{equation}\ker Z_C^{\tr}\Atop[V]=\{0\}.\label{eq:CVcyc2}\end{equation} 
\item[(b)]\label{rem:CVcycles_b} Theorem~\ref{thm:index-1-network-conditions} shows that the index is a structural invariant of the circuit equation. That is, it depends on the interconnection properties of the circuit rather than on parameter values. Notice that our index results are a~slight modification of those in \cite{Estevez-Schwarz_2000aa}, where an index analysis for the modified nodal analysis and charge-oriented  modified nodal analysis has been performed. A~combination of the results from \cite{Estevez-Schwarz_2000aa} with Theorem~\ref{thm:index-1-network-conditions} yields that the circuit DAE \eqref{model1} has index two if, and only if, the MNA equations being subject of \cite{Estevez-Schwarz_2000aa} have index two.
\end{itemize}
\end{remark}

\begin{proof}
 We start with the index result for the DAE~\eqref{eq:mna.cond_b}. To this end notice that the diffeomorphism
 \[    \begin{pmatrix}
    \npot{} \\ {\jcurrv[C]} \\ {\charge[C]} \\ {\flux[L]} \\ {\jcurrv[V]}
    \end{pmatrix}\mapsto \begin{pmatrix}
    \npot{} \\ {u_C} \\  {\jcurrv[L]} \\ {\jcurrv[C]} \\{\jcurrv[V]}
    \end{pmatrix}=\begin{pmatrix}
    \npot{} \\ {\charge^{-1}(\charge[C])} \\ {\flux^{-1}(\flux[L])} \\{\jcurrv[C]} \\ {\jcurrv[V]}
    \end{pmatrix}\]
 applied to the unknown of the DAE \eqref{eq:mna.cond_b} does not change the index, and, by a~suitable permutation of the equations, results in the DAE
\begin{align}
    \nonumber
    \begin{pmatrix}
     0 & 0 & 0 & 0 & 0 \\
     0 & C(u_C) & 0 & 0 & 0 \\
     0 & 0 & L(\jcurrv[L]) & 0 & 0 \\
     0 & 0 & 0 & 0 & 0 \\
     0 & 0 & 0 & 0 & 0 \\
    \end{pmatrix}
    \!\!
    \begin{pmatrix}
    \dot{\npot{}} \\ \dot{u}_C \\  \dot{\jcurrv[L]}\\ \dot{\jcurrv[C]}  \\ \dot{\jcurrv[V]}
    \end{pmatrix}
    & =
    \begin{pmatrix}
         0 &  0 &-\Atop[L] & -\Atop[C] & -\Atop[V] \\
        0 & 0 & 0 & I & 0 \\
        \Atop[L]^{\tr} & 0 & 0 & 0 & 0 \\
        \Atop[C]^{\tr} & -I & 0 & 0 & 0 \\
        \Atop[V]^{\tr} & 0 & 0 & 0 & 0 \\
    \end{pmatrix}
    \begin{pmatrix}
    {\npot{}}  \\ {u}_C\\  {\jcurrv[L]} \\{\jcurrv[C]} \\ {\jcurrv[V]}
    \end{pmatrix}
    \\
    & \label{eq:mna.cond_b_trafo}
    \qquad  -
    \begin{pmatrix}
    \Atop[R]g(\Atop[R]^{\tr}\npot{})\\0\\0\\0\\0
    \end{pmatrix}
+ \begin{pmatrix}
    -\Atop[I] & 0 \\ 0 & 0 \\ 0 & 0 \\ 0 & 0 \\ 0 & - I
    \end{pmatrix}
    \!\!\begin{pmatrix}
    \csrc({\tim}) \\ \vsrc({\tim})
    \end{pmatrix}\! ,
\end{align}
Then Assumption~\ref{ass:soundness-passivity} yields that we are in the situation of \cite[Thm.~6.6]{Rei14}, which yields that the index $\nu$ of \eqref{eq:mna.cond_b_trafo} fulfills
\begin{itemize}
    \item $\nu=0$ if, and only if, the matrix in front of the derivative of the state is invertible. That is, the vectors of potentials, capacitive currents and currents of voltage sources are void.
    \item $\nu=1$ if, and only if, $\nu\neq0$ and
    \begin{align}
        &\ker\begin{pmatrix}
        0&\Atop[R]& -\Atop[C] & -\Atop[V] \\
        C(u_C) & 0 & I & 0 
    \end{pmatrix}^\tr=\{0\}\quad\wedge\quad \label{eq:LI}\\
    &\ker \begin{pmatrix}
        0&0 \\
        0&C(u_C)
    \end{pmatrix}\times\{0\}\times\{0\}=\ker \begin{pmatrix}
        0&0& -\Atop[C] & -\Atop[V] \\
        0&C(u_C)& I & 0 
    \end{pmatrix}\label{eq:CV}
    \end{align}
    \item $\nu=2$ otherwise.
\end{itemize}
The soundness assumption that the circuit has at least one edge implies that the vector of potentials is non-void. Hence, the index of the circuit equations \eqref{eq:mna.cond_b_trafo} is not equalling to zero.\\ 
Further, since \eqref{eq:LI} is equivalent to $(\Atop[C]\, \Atop[R]\, \Atop[V] )$ having full row rank and \eqref{eq:CV}
is equivalent to the full column rank property of $(\Atop[C]\, \Atop[V] )$, we obtain from Proposition~\ref{graphalg} that $\nu=1$ is equivalent to the absence of cycles only consisting of edges to capacitances and/or voltage sources, as well as cuts only consisting of edges to capacitances and/or voltage sources. This completes the proof for the circuit equations
 \eqref{eq:mna.cond_b}.\\
 To prove the index result \eqref{eq:mna.cond_a}, first notice that the characterization for $\nu=0$ follows by the same argumentation as for \eqref{eq:mna.cond_b}. Further notice that a~multiplication of \eqref{eq:mna.cond_b} from the left with a~suitable invertible matrix $T$ and a re-ordering of the state components leads to the DAE
 \begin{align}
    \nonumber
    \ddt\begin{pmatrix}
     \Atop[C] & 0 & 0 & 0 & 0\\
        0 & I & 0 & 0 & 0\\
        0 & 0 & 0 & 0 & 0\\
        0 & 0 & 0 & 0 & 0\\
        0 & 0 & 0 & 0 & 0\\
    \end{pmatrix}
    \!\!
    \begin{pmatrix}
    \charge[C] \\ \flux[L] \\ \npot{} \\ \jcurrv[V] \\u_C    \end{pmatrix}
    & =
    \begin{pmatrix}
         0 & -\Atop[L] & 0 & -\Atop[V] & 0 \\
         \Atop[L]^{\tr} & 0 & 0 & 0 & 0 \\
         0 & 0 & 0 & 0 & 0 \\
        \Atop[V]^{\tr} & 0 & 0 & 0 & 0 \\
        0 & 0 & 0 & 0 & 0 \\
    \end{pmatrix}
    \begin{pmatrix}
    \npot{} \\  \flux^{-1}(\flux[L]) \\ \charge^{-1}(\charge[C]) \\  \jcurrv[V]\\u_C
    \end{pmatrix} 
    \\
    & \label{eq:mna.cond_a_mod}
    \qquad -
    \begin{pmatrix}
    \Atop[R]g(\Atop[R]^{\tr}\npot{})\\0\\ 
        \Atop[C]^\tr \npot - \charge^{-1}(\charge[C]) \\0 \\ \Atop[C]^\tr \npot - u_C
    \end{pmatrix}+ \begin{pmatrix}
    -\Atop[I] & 0  \\ 0 & 0 \\ 0 & 0 \\ 0 & - I \\ 0 & 0
    \end{pmatrix}
    \!\!\begin{pmatrix}
    \csrc({\tim}) \\ \vsrc({\tim})
    \end{pmatrix}\! ,
\end{align}
The upper four equations is exactly the DAE \eqref{eq:mna.cond_a} whereas the variable $u_C$ appears explicitly in the last equation. It can now be inferred from Definition~\ref{def:index} that the index of \eqref{eq:mna.cond_a} does not exceed that of \eqref{eq:mna.cond_b}. By the already proven results for  \eqref{eq:mna.cond_b}, this implies that $\nu\leq2$. Hence it suffices to prove that 
the absence of cycles only consisting of edges to capacitances and voltage sources as well as cuts only consisting of edges to inductances and/or current sources is necessary and sufficient for $\nu\leq1$:\\
To this end, consider matrices $Z_C$, $Z_C'$ with full column rank and $\im Z_C=\ker \Atop[C]^{\tr}$, $\im Z_C'=\im \Atop[C]$. Then $[Z_C\,Z_C']$ is an invertible matrix. 
Now we multiply the first equation in \eqref{eq:mna.cond_a} from the left with $Z_C^{\tr}$ and $(Z_C')^{\tr}$ to obtain an equivalent DAE
\[\begin{aligned}
Z_C^{\tr}\Atop[L]\flux^{-1}(\flux[L])+Z_C^{\tr}\Atop[V]\jcurrv[V]+Z_C^{\tr}\Atop[R]g(\Atop[R]^{\tr}\npot)+Z_C^{\tr}\Atop[I] \csrc(t) 
&=0,\\[1mm]
\ddt(Z_C')^{\tr}\Atop[C]\charge[C]+(Z_C')^{\tr}\Atop[L]\flux^{-1}(\flux[L])+(Z_C')^{\tr}\Atop[V]\jcurrv[V]\qquad&\\+(Z_C')^{\tr}\Atop[R]g(\Atop[R]^{\tr}\npot)+(Z_C')^{\tr}\Atop[I] \csrc(t) 
&=0,\\[1mm]
\ddt \flux[L]-\Atop[L]^{\tr}\npot&=0,\\[1mm]
-\Atop[C]^{\tr}\npot+\charge^{-1}(\charge[C])&=0,\\[1mm]
-\Atop[V]^{\tr}\npot+\vsrc({\tim})&=0.
\end{aligned}\]
The first, forth and fifth equation are now purely algebraic, and will be differentiated in the next step. Using the differentiation rule for inverse functions, we obtain that, for $C$ and $L$ as in Assumption~\ref{ass:soundness-passivity} holds
\[\ddt \charge^{-1}(\charge[C])= C(\charge^{-1}(\charge[C]))^{-1}\ddt \charge[C],\quad \ddt \flux^{-1}(\flux[L])= L(\flux^{-1}(\flux[L]))^{-1}\ddt \flux[L].\]
We further abbreviate $C=C(\charge^{-1}(\charge[C]))$, $L=L(\flux^{-1}(\flux[L]))$ and $G=\frac{\diff g}{\diff u_R}(\Atop[R]^{\tr}\npot)$.
A~differentiation of the algebraic equations now gives
\begin{align*}
&    \nonumber
    \begin{pmatrix}
     0& -Z_C^{\tr}\Atop[L]^{\tr}L^{-1} & -Z_C^{\tr}\Atop[R]G\Atop[R]^{\tr} & -Z_C^{\tr}\Atop[V] \\
    (Z_C')^{\tr} \Atop[C] & 0 & 0 & 0 \\
        0 & I & 0 & 0 \\
        C^{-1} & 0 & \Atop[C]^{\tr} & 0 \\
        0 & 0 & \Atop[V]^{\tr} & 0 \\
    \end{pmatrix}
    \!\!
    \begin{pmatrix}
    \dot{\charge[C]} \\ \dot{\flux[L]} \\ \dot{\npot{}} \\ \dot{\jcurrv[V]}     \end{pmatrix} 
    \\
    & \hspace*{15ex} 
     =
    \begin{pmatrix}
         0 & 0 & 0 & 0 \\
         0 & -(Z_C')^{\tr}\Atop[L] & 0 & -(Z_C')^{\tr}\Atop[V] \\
         \Atop[L]^{\tr} & 0 & 0 & 0 \\
         0 & 0 & 0 & 0 \\
        0 & 0 & 0 & 0 \\
    \end{pmatrix}
    \begin{pmatrix}
    \npot{} \\  \flux^{-1}(\flux[L]) \\ \charge^{-1}(\charge[C]) \\  \jcurrv[V]
    \end{pmatrix} 
    \\
    &
    \hspace{17ex} -
    \begin{pmatrix}
     0 \\
    (Z_C')^{\tr}\Atop[R]g(\Atop[R]^{\tr}\npot{})\\0\\ 
         0  
        \\0
    \end{pmatrix}+ \begin{pmatrix}
    0 & (Z_C')^{\tr}\Atop[I] & 0 \\
    -(Z_C')^{\tr}\Atop[I] & 0& 0  \\ 0 & 0& 0 \\ 0 &0& 0 \\ 0 &0& - I
    \end{pmatrix}
    \!\!\begin{pmatrix}
    \csrc({\tim}) \\\ddt\csrc({\tim}) \\ \ddt\vsrc({\tim})
    \end{pmatrix}\! .
\end{align*}
The definition of the index implies that $\nu\leq1$ if and only if, the matrix in front of the derivative is invertible. By applying elementary row operations to that matrix, we see that 
\begin{equation}
    \nu\leq1\;\Longleftrightarrow \ker\underbrace{\begin{pmatrix}
     0& 0 & -Z_C^{\tr}\Atop[R]G\Atop[R]^{\tr} & -Z_C^{\tr}\Atop[V] \\
    0 & 0 & - (Z_C')^{\tr} \Atop[C]C\Atop[C]^{\tr} & 0 \\
        0 & I & 0 & 0 \\
        C^{-1} & 0 & \Atop[C]^{\tr} & 0 \\
        0 & 0 & \Atop[V]^{\tr} & 0 \\
    \end{pmatrix}}_{=:\widetilde{E}}=\{0\}.\label{eq:mu1}
\end{equation}
If the circuit contains cycles consisting of edges to capacitances and voltage sources or cuts consisting of edges to inductances and/or current sources, then, by Proposition~\ref{graphalg} \& Remark~\ref{rem:CVcycles} $\ker Z_C^{\tr}\Atop[V]\neq\{0\}$. Both lead to $\ker\widetilde{E}\neq\{0\}$ and thus, by \eqref{eq:mu1}, to $\nu>1$.\\
To prove the reverse direction, assume that the circuit neither contains cycles consisting of edges to capacitances and voltage sources nor cuts consisting of edges to inductances and/or current sources. Taking an accordingly partitioned vector $x=(x_1^{\tr}\,x_2^{\tr}\,x_3^{\tr}\,x_4^{\tr})^{\tr}\in\ker\widetilde{E}$, 
we see immediatly that $x_2=0$ holds.
We obtain from the positive definiteness and the fact that $\ker (Z_C')^{\tr}$ equals to the orthogonal complement of $\im \Atop[C]$ that
\[\ker (Z_C')^{\tr} \Atop[C]C\Atop[C]^{\tr}=\ker \Atop[C]^{\tr}.\] Hence, $ x_3 
\in\ker \Atop[C]^{\tr}$, which leads to $x_3=Z_Cw_3$ 
for some real vector $w_3$ 
of suitable size. In particular, $\widetilde{E}x=0$ leads to $\Atop[V]^{\tr}Z_C w_3
=0$, whence $w_3
=Z_{V-C}z_3$ for a~real vector $z_3$ 
and a~matrix $Z_{V-C}$ with full column rank and $\im Z_{V-C}=\ker \Atop[V]^{\tr}Z_C$. A~multiplication of the first row of $\widetilde{E}x=0$ with $Z_{V-C}^{\tr}$ gives, by using $Z_{V-C}^{\tr}Z_C^{\tr}\Atop[V]=0$,
\[0=Z_{V-C}^{\tr}Z_C^{\tr}\Atop[R]G\Atop[R]^{\tr} x_3
=
Z_{V-C}^{\tr}Z_C^{\tr}\Atop[R]G\Atop[R]^{\tr}Z_CZ_{V-C} z_3
\]
and the positive definiteness of $G+G^{\tr}$ (which holds by Assumption~\ref{ass:soundness-passivity}) leads to 
 $\Atop[R]^{\tr}Z_CZ_{V-C} z_3
 =0$. 
 By Proposition~\ref{graphalg} \& Remark~\ref{rem:CVcycles}, the absence of the aforementioned cycles and cuts leads to $\ker (\Atop[C]\,\Atop[R]\,\Atop[V])^{\tr}=\{0\}$ or $\ker Z_C^{\tr}\Atop[V]=\{0\}$.
 The first condition yields $z_3=0$ and thus $x_3=0$, and the second one $x_4=0$. With $x_3=0$, the positive-definiteness of $C$ then finally leads to $x_1=0$. Summing up, we obtain $x=0$, 
 and the index of \eqref{eq:mna.cond_b} equals to one.
\end{proof}
\section{Modeling of coupled circuit DAEs and dynamic iteration schemes}
\label{sec:network-network}

Regarding the coupled circuit DAEs~\eqref{CoupledCircuits} discussed in Section~\ref{sec3.4}, we can take three different perspectives with respect to the input. We will formulate the corresponding circuit equations as PH-DAE systems of type~\eqref{model1}. Note that a~modification of the considerations in this section to the alternative circuit model \eqref{model2} is straightforward. 

Different views on coupled electrical circuits are possible:
\begin{itemize}[leftmargin=6ex]
\item[(C1)] Here all $k$ subsystems, together with the coupling equation, are considered as one system, the PH-DAE system~\eqref{eq.network.coupled.joint}  with state  $\vect{x}^{\tr}:=\bigl(  \charge[C]^{\tr},\,\flux[L]^{\tr},\,\npot[]^{\tr},\,\jcurrv[V]^{\tr},\, \lambdav^{\tr} \bigr)$, and given input $(\csrc^{\tr},\,\vsrc^{\tr})^{\tr}$. \color{black}
\item[(C2)]  We consider the $i$th subsystem separately, with term $\hat u_i = - \lambdav$ arising from the virtual voltage source regarded as an additional input to the system, i.e., the PH-DAE system~\eqref{basis.dynit.network.1tok-1} 
 with state  $\vect{x_i}^{\tr}:=\bigl( \charge[C_i]^{\tr},\, \flux[L_i]^{\tr},\,\npot[i]^{\tr},\, \jcurrv[V_i]^{\tr} \bigr)$, and given input $(\lambdav^{\tr},\csrc^{\tr},\,\vsrc^{\tr})^{\tr}$.\color{black}
\item[(C3)] We consider the $i$th subsystem separately together with the coupling condition, i.e., the PH-DAE system~\eqref{basis.dynit.network.k} with state  $\vect{x_i}^{\tr}:=\bigl( \charge[C_i]^{\tr},\,\flux[L_i]^{\tr},\,\npot[i]^{\tr},\, \jcurrv[V_i]^{\tr},\,\lambdav^{\tr} \bigr)$. Now the vertex potentials $\npot[1],\ldots\npot[k-1]$ add to the input $\hat u_\numsys = \sum_{i=1}^{\numsys-1} \Atop[\lambdav_i]^{\tr} \npot[i]$.\color{red}
\end{itemize}
\color{black}

%

\subsection{Structural properties}

In the following, we investigate the index properties of the $\numsys$ coupled electric circuits, where each subcircuit is assumed to fulfill Assumption~\ref{ass:soundness-passivity}. In particular, each subcircuit is connected and the component matrices have the property that $\vect{C}_i$, $\vect{L}_i$ and $\vect{G}_i+\vect{G}_i^{\tr}$ of each subsystem ($i=1,\dotsc, \numsys$) are pointwise positive definite.\\ \color{black}
We can have different points of view: either regarding the overall system as one joint system or regard just a subsystem with given input, potentially linked to the coupling system or to a part of it. This amounts to certain index assumptions on the overall system (C1) as well as for the subsystems (C2) and (C3). More precisely we will assume that the systems (C1), (C2) and (C3) have index one. Note that,
even in the case that both conditions {(C1)} and {(C2)} are present, condition {(C3)} may not hold.
However, (C3) implies (C1). Of course, it is not a necessary assumption. \color{black}

%


%

%
%


\paragraph{Monolithic perspective.}
For the overall system~\eqref{CoupledCircuits}, the virtual voltage sources extend the set of voltage sources. Thus Theorem~\ref{thm:index-1-network-conditions} yields that the coupled system~\eqref{CoupledCircuits} has index one if, and only if, the circuit neither contains cuts consisting of inductances and/or current sources nor cycles consisting of edges to capacitances and voltage sources. By Proposition~\ref{graphalg} \& Remark~\ref{rem:CVcycles}, this is equivalent to both matrices 
%
%
\begin{equation} \label{eq:global-ACAGAV}
\left(
\begin{pmatrix} \Atop[C_1] & & \\
		& \ddots & \\
		& 	 & \Atop[C_\numsys]
\end{pmatrix},
\begin{pmatrix} \Atop[R_1] & & \\
		  & \ddots & \\
		  & & \Atop[R_\numsys]
\end{pmatrix},
\begin{pmatrix} \Atop[V_1] & &  & \Atop[\lambda_1]\\
		  & \ddots & & \vdots \\
		  & & \Atop[V_\numsys] & \Atop[\lambda_\numsys]
\end{pmatrix}
\right)^{\tr}
\end{equation}
\begin{equation} \label{eq:global-QC-AV}
\begin{pmatrix} \vect{Z}_{C_1}^{\tr} & & \\
		      & \ddots & \\
		      & & \vect{Z}_{C_\numsys}^{\tr}
\end{pmatrix}
\cdot
\begin{pmatrix} \Atop[V_1] & &  & \Atop[\lambda_1]\\
		  & \ddots & & \vdots \\
		  & & \Atop[V_\numsys] & \Atop[\lambda_\numsys]
\end{pmatrix}
\end{equation}
having full column rank.
The latter is equivalent to the full column rank of
\[\displaystyle    \begin{pmatrix}
    \vect{Z}_{V_1-C_1}^{\tr} \vect{Z}_{C_1}^{\tr} \Atop[\lambda_1]  \\
	  \vdots \\
    \vect{Z}_{V_\numsys-C_\numsys}^{\tr} \vect{Z}_{C_\numsys}^{\tr} \Atop[\lambda_\numsys],
\end{pmatrix}
\]
with $\vect{Z}_{C_i}$ and $\vect{Z}_{V_i-C_i}$ being matrices with full column rank and $\im\vect{Z}_{C_i}=\ker \Atop[C_i]^{\tr}$, $\im \vect{Z}_{V_i-C_i}=\ker\Atop[V_i]^{\tr}\vect{Z}_{C_i}$.
\color{black}

\paragraph{Single subsystem perspective.}
We can apply Theorem~\ref{thm:index-1-network-conditions} to the $i$th subsystem~(\ref{SubSystemCircuitsKCL}--\ref{SubSystemCircuitsIdFlux}) to obtain that its index is one if, and only if, the subcircuit neither contains  cuts consisting of inductances and/or current sources nor cycles consisting of edges to capacitances and voltage sources. By Proposition~\ref{graphalg} \& Remark~\ref{rem:CVcycles}, this is equivalent to the full column rank property of the matrices
\begin{align} \label{eq:c2-conditions}
 \vect{Z}_{C_i}^{\tr} \Atop[V_i], \quad
 & (\Atop[C_i],\, \Atop[R_i],\, \Atop[V_i])^{\tr}.  
\end{align}
\color{black}

\paragraph{Subsystem plus coupling equation.}
This DAE has index one if, and only if, 
the subcircuit neither contains cuts consisting of inductances and/or current sources, nor cycles consisting of edges to capacitances together with voltage sources and/or virtual voltage sources. By Proposition~\ref{graphalg} \& Remark~\ref{rem:CVcycles}, this is equivalent to the property that the subsequent two matrices have full column rank:
\begin{align} \label{eq:c3-conditions}
 & (\Atop[C_i],\, \Atop[R_i],\, \Atop[V_i],\, \Atop[\lambda_i])^{\tr}, \quad
    \vect{Z}_{C_i}^{\tr} \left( \Atop[V_i], \Atop[\lambda_i]\right).
\end{align}\color{black}

\subsection{Dynamic iteration perspective on modeling}
\label{sect:dynamic_iteration_perspective}
Dynamic iteration schemes exploit the coupling structure of system~\eqref{CoupledCircuits} by solving subsystems independently and defining a suitable information update.
Let us assume that a numerical approximation 
$(\tilde{\charge[]},\tilde{\flux[]},\tilde{\npot{}},  {\tjcurrv[V]}, \tilde \lambda)$
is given for a time window $[t_{n-1},t_n]$, then a new approximate for the next time window $[t_{n},t_{n+1}]$
can be iteratively derived by the following two steps.
\begin{enumerate}[leftmargin=*]
    \item[i)] Extrapolation step: the approximate solution 
$(\tilde{\charge[]},\tilde{\flux[]},\tilde{\npot{}},  {\tjcurrv[V]}, \tilde \lambda)$
is extrapolated into the  current time window $[t_n,t_{n+1}]$. 
     This defines initial waveforms (approximate solutions)
     $(\charge[]^{(0)}, \flux[]^{(0)}, \npot{}^{(0)}, {\jcurrv[V]}^{(0)}, \lambda^{(0)})$ 
    on $[t_n,t_{n+1}]$ 
     for the following iteration process.
    \item[ii)] Iteration step
    for $l=0,\dotsc,l_{\max}$:
    \begin{enumerate}[leftmargin=*]
        \item[--] The first $\numsys-1$ DAE-IVP subsystems
        (where the constituents are given in~\eqref{basis.dynit.network.1tok-1})
        are solved separately as with respect to the variables 
        \[
            (\charge[i],\flux[i], \npot[i],\jcurrv[V_i]):=
         ({\charge[i]}^{(l+1)},{\flux[i]}^{(l+1)}, {\npot[i]}^{(l+1)},{\jcurrv[V_i]}^{(l+1)}).
         \]
         Thereby the input of the $i$th subsystem is the coupling current 
         $\lambda^{(l)}$; this quantity is given from the previous iteration, i.e., we have $\hat u_i:=-\lambda^{(l)}$. 
      %
\color{black}
 In principle, this could be done in parallel, since these subsystems are decoupled.
\item[--]
The last system (number $\numsys$) can be computed in two different ways:
\begin{itemize}
\item [a)] Jacobi-type approach: 
here one solves
the DAE-IVP~\eqref{basis.dynit.network.k}
with respect to the following variables
\[
    ({\charge[\numsys]},{\flux[\numsys]},{\npot[\numsys]}, {\jcurrv[V_\numsys]},\lambda):=({\charge[\numsys]}^{(l+1)},{\flux[\numsys]}^{(l+1)},{\npot[\numsys]}^{(l+1)}, {\jcurrv[V_\numsys]}^{(l+1)},\lambda^{(l+1)}).
\]
Thereby the input is given by the coupling vertex potentials $\npot[1]^{(l)},\ldots,\npot[\numsys-1]^{(l)}$ from the previous iteration, 
i.e., 
$\hat u_k:=\sum_{i=1}^{\numsys-1} \Atop[\lambda_i]^{\tr} \npot[i]^{(l)}$.
%
In this case, the calculation of the last system could be performed in parallel with the computation of the first $\numsys-1$ systems.

\item [b)] Gauss-Seidel-type approach: 
The only difference to the Jacobi-type approach above is 
the assignment of the input. Here we employ as input the 
coupling vertex potentials $\npot[1]^{(l+1)},\ldots,\npot[\numsys-1]^{(l+1)}$ from the current iteration
instead of the previous one, i.e., we set $\hat u_k:=\sum_{i=1}^{\numsys-1} \Atop[\lambda_i]^{\tr} \npot[i]^{(l+1)}$.
\color{black}
\end{itemize}
\end{enumerate}
    \end{enumerate}

\begin{remark} Notice that
this 
iteration process 
is based on the perspective (C2) for the first $\numsys-1$ subsystems and perspective (C3) for the last subsystem.
\myendofenv
\end{remark}

In the following, 
we will see that the $\numsys$ different subsystems in the dynamic iteration scheme can be interpreted as port-Hamiltonian systems, too. 

\section{Port-Hamiltonian formulation of coupled DAE circuit equations ---  the dynamic iteration perspective}\label{sec:pHsys}

We study the Jacobi approach and the Gauss-Seidel method
for a number of $\numsys$ coupled DAEs.  
To cope with port-Hamiltonian systems arising 
in this context, \color{black}
we have to modify slightly the interconnections.
This is treated in the first part. Secondly, we map the formulation to the electric circuit case. 

\subsection{The dynamic iteration PH-DAE setup}

We give a modified version of Definition~\ref{def:multi-PH-network-DAE}
for the dynamic iteration context. Thereby, we have to introduce the iteration count $l$ and the interconnection needs to map 
outputs of the last iterate to inputs of the current iterate:

\begin{definition}[Multiply coupled PH-DAE---the dynamic iteration perspective]\label{def:multiply-PH-network-dyniter}
We consider  the complete  Definition~\ref{def:multi-PH-network-DAE}
 (multiply coupled PH-DAE with $\numsys$ subsystems)
  apart from the assumption that $\hat{C}$ is skew symmetric. 
We add the iteration count:
the state variables $x_i$, inputs $u_i$ and outputs $y_i$ in \eqref{phdae.coupled.single.network} are labelled with an iteration number $l+1$: $x^{(l+1)}_i,u^{(l+1)}_i,y^{(l+1)}_i$. 
In the case of a Jacobi-type iteration, the $i$th subsystem reads (for $i=1,\dotsc \numsys$) \begin{subequations}\label{phdae.coupled.subsystem.dyniter}
\begin{align}
\ddt E_i  x_i^{(l+1)} 
    =& J_i  z_i^{(l+1)} -r_i\bigl(z_i^{(l+1)}\bigr) + 
    \begin{pmatrix}
    \hat B_i & \bar B_i
    \end{pmatrix}
     \begin{pmatrix} \hat{u}_i^{(l+1)} \\ \bar{u}_i^{(l+1)}  \end{pmatrix} 
    \\
\begin{pmatrix} \hat{y}_i^{(l+1)} \\ \bar{y}_i^{(l+1)}  \end{pmatrix} 
     = &   
    \begin{pmatrix}
    \hat B_i & \bar B_i
    \end{pmatrix}^\tr z_i^{(l+1)}  
     \\
     \intertext{together with the shorthand $z^{(l+1)}= z\left(x_i^{(l+1)}\right)$ 
     and the input (of $i$th subsystem) in the current iteration $(l+1)$ is linked to the output of the previous iteration~$(l)$ by
     }
    0 =& \hat u_i^{(l+1)} + \sum_{j=1,j \neq i}^{\numsys} \hat C_{i,j} \hat y^{(l)}_{j} . \tag{\ref{phdae.coupled.subsystem.dyniter}c-Jacobi}
    \label{eq:coupling-multi-PH-network-dyniter}
 \end{align}
\end{subequations}
\color{black} 
%
%
And we require the Schur complement $\hat{B} \hat{C} \hat{B}^{\tr}$ (of the interconnect matrix $\hat{C}$) to be  skew symmetric.
For the case of a Gauss-Seidel type iteration, only \eqref{eq:coupling-multi-PH-network-dyniter} is replaced 
\color{black} by
\begin{align}
    0 =& \hat u_i^{(l+1)} 
        + \sum_{j=1}^{i-1} \hat C_{i,j} 
            \hat y^{(l+1)}_{j}
        + \sum_{j=i+1}^{\numsys} \hat C_{i,j} \hat y^{(l)}_{j}
    . \tag{\ref{phdae.coupled.subsystem.dyniter}c-GS}
    \label{eq:coupling-multi-PH-network-dyniter-GS}
 \end{align}
\color{black}
\end{definition}

\begin{remark}
In contrast to Definition~\ref{def:multi-PH-network-DAE}, we do not require the interconnection matrix $\hat C$ in~(\ref{phdae.coupled.subsystem.dyniter}c) to be skew-symmetric in the overall.  We only require $\hat{B} (\hat{C}+\hat{C}^{\tr}) \hat{B}^{\tr}=0.$
\myendofenv
 \end{remark}

Now, we have the analogous result to  Corollary \ref{cor.structure.preserving.interconnection.network}:

\begin{corollary}[Multiply skew-symmetric structure-preserving interconnection, Jacobi approach\color{black}]
\label{cor.structure.preserving.interconnection.network.dyniter}
\mbox{}
In the case of dynamic iteration, the assumption of Jacobi-type coupling~\eqref{eq:coupling-multi-PH-network-dyniter} gives
\begin{equation}
\label{eq:phdae.coupled.joint.network.dynit}
\begin{aligned}
\ddt 
\begin{pmatrix}
E & 0 & 0 \\
0 & 0 & 0 \\
0 & 0 & 0
\end{pmatrix}\!\!
\begin{pmatrix}
 x^{(l+1)} \\ {\hat u}^{(l+1)} \\ { \hat y}^{(l+1)}
\end{pmatrix}
 = &
\begin{pmatrix}
J & \hat{B} & 0 \\
- \hat{B}^{\tr} & 0 & I \\
0 & -I &  0
\end{pmatrix}\!\!
\begin{pmatrix}
 z(x^{(l+1)}) \\  \hat u^{(l+1)} \\  \hat y^{(l+1)}
\end{pmatrix}
\!-\!
\begin{pmatrix}
r(z(x^{(l+1)})) \\ 0 \\ 0
\end{pmatrix} \\
& 
+ \begin{pmatrix}
\bar B & 0  \\ 0 & 0 \\ 0 & - \hat C
\end{pmatrix}
\begin{pmatrix}
\bar{u}^{(l+1)} \\ \hat y^{(l)}
\end{pmatrix},
\\
\bar y^{(l+1)}  = &
\begin{pmatrix}
\bar{B}^{\tr} &0 &0 \\
0 & 0 & -\hat C^{\tr}
\end{pmatrix}
\begin{pmatrix}
z(x^{(l+1)}) \\ \hat{u}^{(l+1)} \\ \hat{y}^{(l+1)}
\end{pmatrix} \!. 
\end{aligned}
\end{equation}
which is a PH-DAE 
\begin{align}\label{eq:overall_dyniter_PH}
 \ddt E^{\tot} x^{\tot} &=  J^{\tot} z^{\tot} - r^{\tot}(z^{\tot}) + B^{\tot} u^{\tot}, \\
 \nonumber
 y^{\tot} &= B^{\tot \color{black}\tr} z^{\tot}
\end{align}
with
\begin{align*}
    &\!
    x^{\tot}
        = \begin{pmatrix}
                    x^{(l+1)} \\  \hat u^{(l+1)} \\  \hat y^{(l+1)}
          \end{pmatrix}\!, \;
    z^{\tot} = \begin{pmatrix}  z(x^{(l+1)}) \\  \hat u^{(l+1)} \\  \hat y^{(l+1)}
                              \end{pmatrix}\!\!, \;
    y^{\tot} = \bar{y}^{(l+1)}\!, 
        \;\;\;
    r^{\tot}(z^{\tot}) = \begin{pmatrix} r(z(x^{(l+1)})) \\ 0 \\ 0 \end{pmatrix}\!, 
\\
    &\!
    u^{\tot}
        = \begin{pmatrix} \bar u^{(l+1)} \\ \hat y^{(l)}\end{pmatrix}\!, 
        \;
     E^{\tot}= \!\begin{pmatrix}
            E & 0  & 0\\
            0 & 0 & 0 \\
            0 & 0 & 0
        \end{pmatrix}\!, 
        \;
    J^{\tot}= \!\begin{pmatrix}
            J & \hat{B} & 0 \\
            \!-\hat{B}^{\!\tr\!} & 0 & I \\
            0 & -I &  0
        \end{pmatrix}\!, 
        \;
    B^{\tot}= \!\begin{pmatrix}
                \bar B &0 \\
                0 &0 \\
                0 & -\hat C
            \end{pmatrix}\!.
\end{align*}
For the Gauss-Seidel coupling~\eqref{eq:coupling-multi-PH-network-dyniter-GS}, a  PH-DAE~\eqref{eq:overall_dyniter_PH} can be formulated  with 
\begin{align*}
 J^{\tot}= \!\begin{pmatrix}
            J & \hat{B} & 0 \\
            \!-\hat{B}^{\!\tr\!} & 0 & I \\
            0 & -I &  -\hat C
        \end{pmatrix}\!, 
        \;
    B^{\tot}= \!\begin{pmatrix}
                \bar B &0 \\
                0 &0 \\
                0 & \hat C_1
            \end{pmatrix}\!,
            \;
  u^{\tot}
        = \begin{pmatrix} \bar u^{(l+1)} \\  \Delta \hat y^{(l+1)}\end{pmatrix}\!           
\end{align*}
instead of $J^{\tot}$,  $B^{\tot}$ and $u^{\tot}$ above,
provided that $\hat C$ is skew-symmetric.
Here we have used the short-hand $\hat u^{(l+1)} + C_1 \hat y^{(l)} + C_2 \hat y^{(l+1)}=0$.
\end{corollary}

\begin{remark}
\label{rem.change.ham.coupled.network.dyniter}
\begin{enumerate}[label=\roman*)]
\item 
The only difference to the setting of Corollary~\ref{cor.structure.preserving.interconnection.network} is the following: $\hat y^{(l)}$ defines a new input variable, and correspondingly, the coupling matrix $\hat C$  (Jacobi-type approach) and $\hat C_1$ (Gauss-Seidel type approach), resp.,  is shifted from the structure matrix $J^{\tot}$ to the port matrix $B^{\tot}$. 

\item In the dynamic iteration case~\eqref{eq:phdae.coupled.joint.network.dynit}, 
 Jacobi-type approach, 
the system can be condensed to
%
\begin{subequations}
\label{eq:phdae.coupled.joint.network.condensed.dynit.v3}
\begin{align}
\ddt E x^{(l+1)} & 
    =  \hat J z(x^{(l+1)})  - r(z(x^{(l+1))})  
        + \hat B \hat C \hat B^{\tr} \Delta z^{(l+1)}
        + \bar B \bar u^{(l+1)},
    \\
\bar{\bar{y}}^{(l+1)}  & 
    =  \left(\hat B \hat C \hat B^{\tr}\right)^{\tr} z(x^{(l+1)})
        \; =\;  -  \left(\hat B \hat C \hat B^{\tr}\right) z(x^{(l+1)}),
        \label{eq:phdae.coupled.joint.network.condensed.dynit-b-v3} 
    \\
\bar y^{(l+1)} & 
    = \bar B^{\tr} z(x^{(l+1)})
\end{align}
\end{subequations}
with $\hat{J}= J-\hat{B} \hat{C} \hat{B}^{\tr}$
and with  an extra 
output $\bar{\bar{y}}^{(l+1)}= -\color{black} \hat{B}\hat{C}\hat{B}^{\tr}  \hat{y}^{(l+1)} $.
Moreover, we note that $(\bar{\bar{u}}^{(l+1)}:=)\, \Delta z^{(l+1)} = z^{(l+1)} -z^{(l)} $ 
\color{black}
is the dynamic iteration update and it takes the role of an extra input. 
Note that in the Gauss-Seidel-type approach, the same PH-DAE~\eqref{eq:phdae.coupled.joint.network.condensed.dynit.v3}
holds, with $\hat B \hat C \hat B^{\tr}$ replaced by $\hat B \hat C_1 \hat B^{\tr}$.%
%
%

\item 
     Here,   the change in the Hamiltonian is given by
    \begin{equation}\label{eq:hamiltonian-dyniter}
\begin{aligned}
& - \int_t^{t+h} z(x^{(l+1)}(\tau))^{\tr} r\left(z(x^{(l+1)}(\tau)\right) \, d\tau + \\
& \int_t^{t+h} \left(  {\bar{u}^{(l+1)}(\tau)}^{\tr}
\bar B^{\tr}
z(x^{(l+1)}(\tau))  - (\Delta z(x^{(l+1)}(\tau))^{\tr} \hat B \hat C \hat B^{\tr} z(x^{(l)}(\tau))
 \right) \, d\tau.
 \end{aligned}
\end{equation}
We point out that the third term in~\eqref{eq:hamiltonian-dyniter}, which is additional to the first two terms already known from~\eqref{change.hamiltonian.condensed}, decays with the converging dynamic iteration.
\item For the use of a Gauss-Seidel iteration in Corollary~\ref{cor.structure.preserving.interconnection.network.dyniter}, the input $\hat{y}^{(l)}$ needs to be split into old and new iterates. 
%
%
\myendofenv
\end{enumerate}
\end{remark}

\subsection{The dynamic iteration perspective for multiply coupled electric circuits}
%
%
We study a number of $\numsys$ coupled  circuits, which were given in charge oriented form in~\eqref{CoupledCircuits}. 
In the perspective of $\numsys$ copies of the PH-DAE model from  Proposition~\ref{prop:hamiltonian-network}, the respective constituents are already given in the proofs of Lemma~\ref{lem:PH-DAE_mutally_coupled_networks} (in~\eqref{basis.dynit.network.1tok-1} for the systems $1,\,\dotsc,\,\numsys-1$, and of~\eqref{basis.dynit.network.k} for system $\numsys$).
Now, we include  
\color{black}
the dynamic iteration process. 
First, in the $l+1$-st iteration, say, we solve (successively or in parallel)
the subsystems $i=1,\dotsc, \numsys-1$. 
These subsystems read in the PH-DAE notation (cf.~Corollary~\ref{cor.structure.preserving.interconnection.network.dyniter}) for both the Jacobi and the Gauss-Seidel approach as follows: 
\begin{subequations}\label{eq:network-dynit.1tok-1}
 \renewcommand{\arraystretch}{1.05}
 \begin{align}
    \nonumber
    \ddt &\begin{pmatrix}
     \Atop[C_i] & 0 & 0 & 0 \\
        0 & I & 0 & 0 \\
        0 & 0 & 0 & 0 \\
        0 & 0 & 0 & 0 \\
    \end{pmatrix}
    \!\!
    \begin{pmatrix}
    \charge[_i]^{(l+1)} \\ \flux[_i]^{(l+1)} \\ \npot{}_i^{(l+1)} \\ \jcurrv[V_i]^{(l+1)}    \end{pmatrix}
     =
    \begin{pmatrix}
         0 & -\Atop[L_i] & 0 & -\Atop[V_i] \\
         \Atop[L_i]^{\tr} & 0 & 0 & 0 \\
         0 & 0 & 0 & 0 \\
        \Atop[V_i]^{\tr} & 0 & 0 & 0 \\
    \end{pmatrix}
    \begin{pmatrix}
    \npot{}_i^{(l+1)} \\  \jcurrv[L_i]^{(l+1)} \\ u_{C_i}^{(l+1)} \\  \jcurrv[V_i]^{(l+1)}
    \end{pmatrix} 
    \\ 
    & 
    \qquad 
    \color{magenta} - \color{black}
   \begin{pmatrix}
   \Atop[R_i]g_i(\Atop[R_i]^{\tr} \npot{}_i^{(l+1)})\\          0\\ 
        \Atop[C_i]^\tr \npot{}_i^{(l+1)} - u_{C_i}^{(l+1)} \\
            0
    \end{pmatrix}
    + 
    \begin{pmatrix}
     \Atop[\lambda_i] \\ 0 \\ 0 \\ 0
    \end{pmatrix}
    {\hat u_i} 
    + \begin{pmatrix} 
    -\Atop[I_i] & 0  \\ 0 & 0 \\ 0 & 0 \\ 0 & - I
    \end{pmatrix}
    \!\!\begin{pmatrix}
    \csrc[i]({\tim}) \\ \vsrc[i]({\tim})
    \end{pmatrix}\! ,
    \\
    & \quad \hat y_i^{(l+1)}  =
    \begin{pmatrix}
     \Atop[\lambda_i] \\ 0  \\ 0 \\ 0
    \end{pmatrix}^{\tr}
    \begin{pmatrix}
    \npot{}_i^{(l+1)} \\  \jcurrv[L_i]^{(l+1)}\\ 
    u_{C_k}^{(l+1)} \\ \jcurrv[V_i]^{(l+1)}
    \end{pmatrix} ,
    \qquad
    \qquad \bar y_i  =  \begin{pmatrix}
    -\Atop[I_i] & 0  \\ 0 & 0 \\  0 & 0 \\ 0 & - I_i
    \end{pmatrix}^{\tr}
    \begin{pmatrix}
    \npot{}_i^{(l+1)} \\  \jcurrv[L_i]^{(l+1)}\\ 
    u_{C_k}^{(l+1)} \\ \jcurrv[V_i]^{(l+1)}
    \end{pmatrix},
\end{align}
\end{subequations}
where the inputs are connected to the output of the last system (number $\numsys$) from the previous iteration step $(l)$:
(for both approaches)
\begin{equation}
\label{input.output.jacobi+gauss-seidel}
\hat{u}_i=\hat{u}_i^{(l)} =-\hat y_k^{(l)}, \qquad i=1,\dotsc, \numsys-1
\end{equation}
(and it is used within $z_i$:   $\jcurrv[L_i]^{(l+1)} =\phi^{-1}_{i} (\phi_{L,i}^{(l+1)})$, $u_{C,i}^{(l+1)} =q^{-1}_{i} (q_{C,i}^{(l+1)})$). 
Finally, the $\numsys$-th subsystem (last) 
reads for both approaches
\begin{subequations}\label{eq:network-dynit.k}
\begin{align}
    \nonumber
    \ddt& \begin{pmatrix}
     \Atop[C_k] & 0 & 0 & 0 & 0 \\
        0 & I & 0 & 0 & 0\\
        0 & 0 & 0 & 0 & 0\\
        0 & 0 & 0 & 0 & 0\\
         0 & 0 & 0 & 0 & 0
    \end{pmatrix}
    \!\!
    \begin{pmatrix}
    \charge[\numsys]^{(l+1)} \\ 
    \flux[\numsys]^{(l+1)} \\ 
    \npot{}_{\numsys}^{(l+1)} \\ 
    \jcurrv[V_\numsys]^{(l+1)} \\
    \lambdav_\numsys^{(l+1)}
    \end{pmatrix}
    =
    \begin{pmatrix}
         0 & -\Atop[L_k] & 0 & -\Atop[V_k] & -\Atop[\lambda_k]\\
         \Atop[L_k]^{\tr} & 0 & 0 & 0 & 0\\
         0 & 0 & 0 & 0 & 0\\
        \Atop[V_k]^{\tr} & 0 & 0 & 0 & 0\\
        \Atop[\lambda_k]^{\tr} & 0 & 0 & 0 & 0
    \end{pmatrix}
    \begin{pmatrix}
    \npot{}_{\numsys}^{(l+1)} \\  
    \jcurrv[L_\numsys]^{(l+1)} \\ 
    u_{C_k}^{(l+1)} \\  
    \jcurrv[V_k]^{(l+1)} \\
     \lambdav_k^{(l+1)}
    \end{pmatrix} 
    \\ 
    & 
    \qquad 
    \color{magenta} - \color{black}
   \begin{pmatrix}
    \Atop[R_k]g_k( \Atop[R_k]^{\tr}
            \npot{}_k^{(l+1)} )\\
                0\\ 
        \Atop[C_k]^\tr \npot{}_k^{(l+1)} 
        - u_{C_k}^{(l+1)} \\
                0 \\
                0
    \end{pmatrix}
    + \begin{pmatrix}
    0 \\ 0 \\ 0 \\ 0 \\ I
    \end{pmatrix}
    {\hat u_k} 
    + \begin{pmatrix} 
    -\Atop[I_k] & 0  \\ 0 & 0 \\ 0 & 0 \\ 0 & - I \\ 0 & 0 
    \end{pmatrix}
    \!\!\begin{pmatrix}
    \csrc[k]({\tim}) \\ \vsrc[k]({\tim})
    \end{pmatrix}\! ,
     \\
     & \;\,
     \hat y_\numsys^{(l+1)}  =
    \begin{pmatrix}
     0 \\ 0 \\  0 \\  0 \\ I
    \end{pmatrix}^{\tr}
    \begin{pmatrix}
        \npot{}_\numsys^{(l+1)} \\  
        \jcurrv[L_\numsys]^{(l+1)} \\  
        u_{C_k}^{(l+1)} \\ 
        \jcurrv[V_\numsys]^{(l+1)} \\
        \lambdav^{(l+1)}
    \end{pmatrix}\!,
    \qquad\qquad
    \bar y_\numsys  =
    \begin{pmatrix}
        -\Atop[I_\numsys] & 0  \\ 0 & 0 \\ 0 & 0  \\ 0 & - I_\numsys  \\ 0 & 0
    \end{pmatrix}^{\tr}
    \begin{pmatrix}
        \npot{}_\numsys^{(l+1)} \\  
        \jcurrv[L_\numsys]^{(l+1)} \\  
        u_{C_k}^{(l+1)} \\ 
        \jcurrv[V_\numsys]^{(l+1)} \\
        \lambdav^{(l+1)}
    \end{pmatrix}.   
\end{align}
\end{subequations}
Only, the relation of outputs and inputs differs:
for the Jacobi case, we have a relation to the previous iterates:
\refstepcounter{equation}\label{eq:input.output.k-set}
\begin{equation}
\label{input.output.k.jacobi}
  \hat{u}_k=\hat{u}_k^{(l)} \color{black}
     =\sum_{i=1}^{k-1} \hat y_i^{(l)};
     \tag{\ref{eq:input.output.k-set}-Jacobi}
\end{equation}
and in the Gauss-Seidel case, the current iterates need to be used:
\begin{equation}
\label{input.output.k.gauss-seidel}
   \hat{u}_k=\hat{u}_k^{(l+1)} \color{black}
   =\sum_{i=1}^{k-1} \hat y_i^{(l+1)}.
     \tag{\ref{eq:input.output.k-set}-GS}
\end{equation}

In both cases, after aggregation, the $\numsys$ subsystems can be written as a joint PH-circuit-DAE system.

\begin{mylemma}\label{lem:Jacobi-overall}
For the Jacobi  approach, system \eqref{eq:network-dynit.1tok-1}$+$\eqref{eq:network-dynit.k} with both input-output relation \eqref{input.output.jacobi+gauss-seidel}$+$\eqref{input.output.k.jacobi} is in the overall a PH-DAE of type~\eqref{eq:phdae.coupled.joint.network.dynit}.
\end{mylemma}

\begin{proof}
 We identify via aggregation the terms in~\eqref{eq:phdae.coupled.joint.network.dynit}:
 \begin{align*}
    & x^{(l+1)} = \begin{pmatrix}
    \charge[]^{(l+1)} \\ \flux[]^{(l+1)} \\\npot{}^{(l+1)} \\    \jcurrv[V]^{(l+1)} \\ \lambdav^{(l+1)}
    \end{pmatrix} 
    \!, \quad 
    \bar{u}^{(l+1)} =  \begin{pmatrix}
    \csrc({\tim}) \\ \vsrc({\tim})
    \end{pmatrix} 
    \!, 
    \quad
    \hat{u}=\hat{u}^{(l)} = \begin{pmatrix} \hat{u}_1^{(l)} \\
    \vdots \\ \hat{u}_k^{(l)}\end{pmatrix}
    %
    \! , 
    \quad
  \hat{y}^{(l+1)}
        = \begin{pmatrix} \hat{y}_1^{(l+1)} \\
            \vdots \\ \hat{y}_k^{(l+1)}
            \end{pmatrix} 
            \! , 
    \\
    &
    E=  \begin{pmatrix} \Atop[C]  & 0 & 0& 0 &0 \\
        0 & I & 0 & 0 & 0 \\
        0 & 0 & 0 & 0 & 0\\
        0 & 0 & 0 & 0 & 0 \\
        0 & 0 & 0 & 0 & 0
    \end{pmatrix}\!, 
    \quad 
    r
    (z^{(l+1)}) =  \begin{pmatrix}
        \Atop[R] \resi[R](\Atop[R]^{\tr} \npot{}^{(l+1)}) \\
            0 \\
            \Atop[C]^{\tr} \npot[] - u_C \\
            0 \\
            0 
    \end{pmatrix},
 \\
    &
    J = \begin{pmatrix}
        0 & - \Atop[L] & 0 & - \Atop[V] & - \tilde{\Atop[\lambda]}  \\
        \Atop[L]^{\tr} & 0 & 0 & 0 & 0 \\
        0 & 0 & 0 & 0 & 0 \\
        \Atop[V]^{\tr} & 0 & 0 & 0 & 0 \\
        \tilde{\Atop[\lambda]}^{\tr} & 0 & 0 & 0 & 0 
      \end{pmatrix}\!,
        \quad\!\!
        %
    \tilde{\Atop[\lambda]}
        =\begin{pmatrix} 0 \\ \vdots \\  0 \\ \Atop[\lambda_k]\end{pmatrix} 
        \! , 
  \quad
    \hat B =
    \begin{pmatrix}
        \hat{\Atop[\lambda]}  \\
          0 \\ 
          0 \\
         0 \\
         F
    \end{pmatrix} 
    \! , 
 \\
    &
    \hat{\Atop[_\lambda]} = \blkdiag \begin{pmatrix}
        \Atop[\lambda_1], \dotsc, \Atop[\lambda_{k-1}], 0
    \end{pmatrix}
    =
  \begin{pmatrix} \Atop[\lambda,1] \\
    & \Atop[\lambda, 2] \\
    & & \ddots \\
    & & & \Atop[\lambda, \numsys-1] \\
    & & & & 0\end{pmatrix} 
    \! , 
     \quad
    \\
&
    F = (0,\, \dotsc,\, 0, \, I)
         , \quad
   \bar B\!=\!\begin{pmatrix}
                 - \Atop[I] & 0 \\
                0 & 0 \\
                0 & 0 \\
                0 & -  I  \\
                0& 0
            \end{pmatrix}\!  
            ,
           \quad
   \hat{C}\!=\!\begin{pmatrix}
                0 & \cdots & 0 & I_{n_\lambda} \\
                \vdots & & \vdots & \vdots \\
             0 & \cdots &0 & I_{n_\lambda} \\
            -I_{n_\lambda} & \cdots & -I_{n_\lambda} & 0
            \end{pmatrix} \!
            ,
\end{align*}
where $F$ is split analogously to $\hat{\Atop[\lambda]}$.
%
\end{proof}

\begin{remark} 
i) Note that the change in the Hamiltonian according to~\eqref{eq:hamiltonian-dyniter} is given by
\begin{align}
&    
\int_t^{t+h} 
 \left(
 - \npot^{(l+1)}(\tau)^{\tr} \Atop[R] g(\Atop[R]^{\tr} \npot^{(l+1)}(\tau))
 - 
    \csrc(\tau)^{\tr} \Atop[I]^{\tr} \npot{}^{(l+1)}(\tau)- \vsrc(\tau)^{\tr} \jcurrv[V]^{(l+1)}(\tau)
      \right. +\nonumber \\
    & \qquad + \left. \sum_{i=1}^{\numsys-1} \left[
 {(\Delta {\npot{}_i}^{(l+1)})}^{\tr}
\Atop[\lambda_i] \lambdav^{(l+1)} -
{(\Delta \lambdav^{(l+1)})}^{\tr}
{(\Atop[\lambda_i])}^{\tr} {\npot{}_i}^{(l+1)}
    \right]
    \right) \, d\tau.
\end{align} 
ii) The Schur complement part for the condensed version, cf.~\eqref{eq:phdae.coupled.joint.network.condensed.dynit.v3},
reads: 
\[
    \hat B \hat C \hat B^{\tr} \!=\!
\begin{pmatrix}
0 & 0 & 0 & 0 & \bar{\Atop[\lambda]} \\
0 & 0 & 0 & 0 & 0  \\
0 & 0 & 0 & 0 & 0  \\
0 & 0 & 0 & 0 & 0  \\
-\bar{\Atop[\lambda]}^{\tr} & 0 & 0 & 0 & 0
    \end{pmatrix} \quad \text{ with }
    \quad
    \bar{\Atop[\lambda]}^{\!\!\tr} = ({\Atop[\lambda_1]}^{\!\!\tr},\,
    \dotsc,\, {\Atop[\lambda_{\numsys-1}]}^{\!\!\tr},\,  0).
\]

iii) We state explicitly the matrix $\hat{B}$ with dimensions:
  \renewcommand{\kbldelim}{(}
\renewcommand{\kbrdelim}{)}
\[
  \hat{B} =
\kbordermatrix{
        & n_\lambda & \cdots & n_\lambda & n_\lambda \\
    n_{e,1} & \Atop[\lambda,1] &  &  &  \\
    \vdots  &  & \ddots &  &  \\
    n_{e,\numsys-1} &  &        &  \Atop[\lambda,\numsys-1] &  \\
    n_{e,\numsys} &  &  &  & 0  \\[2ex]
    n_{L,1}    & 0 & \cdots & \cdots & 0 \\
    \vdots    & \vdots     & &  &\vdots \\
    n_{L,\numsys}    &  0 &  \cdots &  \cdots &  0 \\[2ex]
     n_{u_C,1}    & 0 & \cdots & \cdots & 0 \\
     \vdots    &  \vdots     & &  & \vdots \\
    n_{u_C,\numsys}    &  0 &  \cdots &  \cdots &  0 \\[2ex]
    n_{V,1}    & 0 & \cdots & \cdots& 0 \\
     \vdots   & \vdots     & &  &\vdots \\
    n_{V,\numsys} & 0 & \cdots & \cdots& 0 \\[2ex]
     n_\lambda   & 0 & \cdots & 0 & I
  }
\]
\myendofenv
\end{remark}

\begin{remark}
From remark~\ref{rem.change.ham.coupled.network.dyniter}(ii) we know that in the case of Gauss-Seidel type iteration, the condensed PH-DAE description reads 
\begin{align*}
    &\ddt E {x}^{(l+1)} = \hat J z^{(l+1)}
        -  r(z^{(l+1)}) + \hat B^{GS}  \Delta \lambda + \bar{B} \bar{u}
    \\
    & {\bar{\bar y}}^{(l+1)} = \hat{B}^{GS,\tr} z^{(l+1)} \quad \left(= -\Atop[\lambda,k]^{\tr} e_k^{(l+1)} \right)
    \\
    & \bar{y}^{(l+1)}= \bar{B}^{\tr} z^{(l+1)}
\end{align*}
where we have used 
\begin{align*}
  \hat B \hat C_1 \hat B^{\tr} \Delta z^{(l+1)}= 
  \underbrace{\begin{pmatrix}
  \bar{\Atop[\lambda]} \\ 0 \\ 0  \\ 0  \\ 0 
  \end{pmatrix}}_{\displaystyle \hat B^{GS}:=} \Delta \lambda^{(l+1)}.
\end{align*}

The error in the Hamiltonian is given by
\begin{align}
& -
\int_t^{t+h} \left( \npot[]^{(l+1}(\tau)^{\tr} \Atop[R] g(\Atop[R]^{\tr} \npot[]^{(l+1)}(\tau)) + \csrc[]^{\tr} (\tau) \Atop[I]^{\tr} x^{(l+1)}(\tau) +
\vsrc[]^{\tr}(\tau)\jcurrv[V]^{(l+1)}(\tau) +  \right. 
\nonumber \\
& + \left. (\Delta \lambda^{(l+1)}(\tau))^{\tr} \Atop[\lambda_k]^{\tr} \npot{}_\numsys^{(l+1)}(\tau)
\right)  \, d \tau.
\end{align}
\end{remark}



\section{Conclusions}

We have introduced several PH-DAE formulations, where all cases correspond to dedicated perspectives: overall systems, multiply coupled DAE systems, and  systems within a dynamic iteration process. We proved  that versions of the charge-oriented electric circuit models (based on MNA) fall into these classes. Furthermore,  we showed that dynamic iteration processes of such PH-DAE systems yield in a certain setup again PH-DAEs. The splitting error enters the respective Hamiltonian as an additional term. 

In particular, we  included  nonlinear dissipative terms in the PH-DAE setup and we added DAE specific subspace restrictions. Moreover, dissipativity of electric circuits is here treated very generally by assuming the existence of according gradient fields. A discussion on structural properties (in our case with respect to the differential index) reveals that known index results translate to our new PH-DAE settings. 

We believe that our concepts of PH-DAEs can be applied also to other DAEs, in particular to DAEs stemming from multibody systems and flow networks.

The next steps include the development of discretizations, which respect the PH-DAE structure and preserve in this way the energy  in order to enable fully discrete systems with the same properties. 

\bibliography{bartel}{}
\bibliographystyle{abbrv}
\end{document}